\documentclass[a4paper,twoside]{article}
\usepackage[Explictly]{optional}

\usepackage{geometry}
\AtEndDocument{\bigskip{\footnotesize%
		\textsc{Department of Mathematics, National Taiwan normal University} \par  
		\textit{E-mail address}, Yu Tong Liu: \texttt{80940004S@ntnu.edu.tw} \par
	}}
\usepackage{clipboard}

\usepackage{import}
\usepackage{amssymb}
\usepackage{amsmath}
\usepackage{amsthm}

\usepackage{mathrsfs}
\usepackage{hyperref}
\usepackage{cleveref}

\usepackage{enumitem}

\usepackage[square,numbers]{natbib}
\usepackage{url}
\usepackage{mathtools,tikz-cd}
\usepackage{esint}                \usepackage{upgreek}
\usepackage{xcolor}
\usepackage{xtab}

\newtheoremstyle{citing}{3pt}{3pt}{\itshape}{}{\bfseries}{.}{.5em}{\thmnote{#3}}\theoremstyle{citing}

\theoremstyle{plain}
\newtheorem{introTheorem}{Theorem}

\swapnumbers
\theoremstyle{plain}
\newtheorem{theorem}{Theorem}[section]

\newtheorem{lemma}[theorem]{Lemma}
\newtheorem{corollary}[theorem]{Corollary}

\theoremstyle{remark}
\newtheorem{remark}[theorem]{Remark}
\newtheorem{example}[theorem]{Example}
\theoremstyle{definition}
\newtheorem{definition}[theorem]{Definition}
\newtheorem{defparagraph}[theorem]{}
\newtheorem{propparagraph}[theorem]{}

\newcounter{counter1}
\newcounter{counter2}

\newcommand{\integers}{\integers}

\newcommand{\weight}{\mathbin{\vrule height 1.6ex depth 0pt width
		0.14ex\vrule height 0.14ex depth 0pt width 1.3ex}}

\newcommand{\NormBoundedExpr}[3]{#1(#2)\leq 1, #2 \in #3 }

\DeclareMathOperator{\card}{card}

          \DeclareMathOperator{\spt}{spt}     \DeclareMathOperator{\im}{im}       \DeclareMathOperator{\Int}{Int}     \DeclareMathOperator{\diam}{diam}        \DeclareMathOperator{\dmn}{dmn}     \DeclareMathOperator{\dist}{dist}   \DeclareMathOperator{\Hom}{Hom}                     \DeclareMathOperator{\pt}{pt}
\DeclareMathOperator{\D}{D}

\DeclareMathOperator{\Lap}{Lap}
\newcommand{\RomanNumeralCaps}[1]
{\MakeUppercase{\romannumeral #1}}

\newcommand{\BourbakiTVS}[1]{\cite[#1]{MR910295}}

\newcommand{\BourbakiGTI}[1]{\cite[#1]{bourbaki1998general1}}
\newcommand{\BourbakiGTII}[1]{\cite[#1]{bourbaki1998general2}}
\newcommand{\mathuniformity}[1]{\mathfrak{#1}}

\newcommand{\BourbakiCA}[1]{\cite[]{}}
\newcommand{\BourbakiAlg}[1]{\cite[]{bourbaki1998algebra,bourbaki2013algebra}}
\newcommand{\BourbakiLieAlg}[1]{\cite[]{}}
\newcommand{\BourbakiSet}[1]{\cite[]{}}
\newcommand{\BourbakiCiteNumber}[3]{\RomanNumeralCaps{#1}.\S #2.#3}

\begin{document}

\author{
	Yu Tong
	\footnote{National Taiwan Normal university, Department of Mathematic, Taipei, Taiwan, 80940004S@ntnu.edu.tw}
}
\title{Pointwise differentials for Distributions}
\maketitle

\begin{abstract}
	The notion of pointwise differentials for distributions is a way to
	extract local information of distributions by rescaling the distribution at a point. 
	In this paper, we study the pointwise differentials
	for distributions corresponding to
	a  negative order Sobolev functions.
	Our main results prove Borel regularity, Lusin approximation, rectifiability,
	and a Rademacher theorem for these pointwise differentials.
\end{abstract}
\tableofcontents

\section{Introduction}

This paper continues the study of pointwise differentials for distributions. 
The theory of pointwise differentials for functions in Lebesgue spaces was introduced by Calder\'on and Zygmund in  \citep{MR0136849} to study the local properties of strong solutions of elliptic partial differential equations. Roughly speaking, they studied the asymtotic behavior of rescalings of a function $f$ compared with a polynomial $P$ using $\mathbf L_p$ norm. Motivated by regularity questions in Geometric Measure Theory, 
in \cite{2021} Menne introduced the notion of pointwise differentials for weak solutions, using the same idea as in \citep{MR0136849} to study the local behavior of a distribution in \citep{2021} but with different norm. The pointwise differential for distributions has occurred in \citep{MR3023856}, where Rademacher-Stepanov type results were obtained as an intermediate step of the regularity result of intergal varifold with locally bounded variation. The goal of this paper is to study the pointwise differential for distribution using  Sobolev norms by proving four main results as in \citep{2021}. Before we can state our results, we need the following definitions and notation.

Below, we recall the notion of pointwise differentials for distributions introduced in \cite{2021}. See section \ref{sec:Notation} for definitions of polynomial functions \ref{symmetri tensor} and distributions \ref{defpar:Distribution}.
In particular, given a Banach space $Y$, 
we denote $Y'$ to be the topological dual of $Y$ (Note that in \citep{2021},
 $Y^\ast$ is the topological dual of $Y$), 
  $\mathscr D(U,Y)$ denote the vector space of smooth functions on
   open subset $U$ of $\mathbf R^n$ with valued in $Y$, 
    $\mathscr D_K(U,Y)$ denote the vector subspace consist of 
	$f\in \mathscr D(U,Y)$ with support in $K$,
	 denote $\mathscr D'(U,Y)$ to be the set of distributions of $Y$ in $U$.
	  Whenever 
	  $\theta \in \mathscr D_{\mathbf B(0,1)}(\mathbf R^n,Y),T\in \mathscr D'(\mathbf R^n,Y),a\in \mathbf R^n,0<r<+\infty$,
	   we write $T^{a,r}(\theta)=r^{-n}T_x(\theta (r^{-1}(x-a))), \boldsymbol{\nu}_{\mathbf B(0,1)}^i(\theta) 
	   = \sup \im \|\D^i \theta\|,\nu_{i,p}(\theta)=\left(\int \|\D^i \theta\|^p \ d \mathscr L^n\right)^{1/p}$. 
	   Throught out this paper, $n$ will be a positive integer.
 \begin{definition}\cite[2.23]{2021}\label{IntroDef:Pt} \
 	Suppose $Y$ is a Banach space, $T\in \mathscr D'(\mathbf R^n, Y)$, $\nu$ is a seminorm on $\mathscr D_{\mathbf {B}(0,1)}(\mathbf R^n,Y)$, $k$ is  an integer, $0 < \alpha \leq 1$, $k+\alpha \geq 0$, and $a\in \mathbf R^n$.	
 	\begin{enumerate}
 		\item The distribution $T$ is  pointwise differentiable of order $(k,\alpha)$ at $a$ if and only if there exists a polynomial function $P : \mathbf R^n \to Y'$  polynomial  of degree at most $k$ such that 
 		\[ \limsup_{r\to0^+} r^{-k-\alpha}(T - P)^{a,r}(\theta) <+\infty \quad \text{for every $\theta \in \mathscr D_{\mathbf{B}(0,1)}(\mathbf R^n,Y)$ }.\]	
 		\item The distribution $T$ is  pointwise differentiable of order $k$ at $a$ if and only if there exists a polynomial function $P : \mathbf R^n \to Y'$ of degree at most $k$, such that 
 		\[ \lim_{r\to0^+} r^{-k}(T - P)^{a,r}(\theta) =0 \quad \text{for every $\theta \in \mathscr D_{\mathbf{B}(0,1)}(\mathbf R^n,Y)$ }.\] 
 		If $T$ is pointwise differentiable of order $k\geq 0$ at $a\in \mathbf R^n$, then $P$ is the unique polynomial map of degree at most $k$ satisfying the above condition, and for every integer $0\leq m \leq k$, we define  $\pt \D^m T(a)= \D^m P(a)$. The polynomial map $P$ is a called the \textit{$k$ jet} of $T$ at $a$.
 		\item The distribution $T$ is $\nu$ pointwise differentiable of order $(k,\alpha)$ at $a$ if and only if there exists a polynomial function $P : \mathbf R^n \to Y'$  of degree at most $k$ such that 
 		\[ \limsup_{r\to0^+} \sup\{ r^{-k-\alpha}(T - P)^{a,r}(\theta) : \NormBoundedExpr{\nu}{\theta}{\mathscr D_{\mathbf{B}(0,1)}(\mathbf R^n,Y)} \}< +\infty.\]
 		\item The distribution $T$ is $\nu$ pointwise differentiable of order $k$ at $a$ if and only if there exists a polynomial function $P : \mathbf R^n \to Y'$ of degree at most $k$, such that 
 		\[ \lim_{r\to0^+} \sup\{ r^{-k}(T-P)^{a,r}(\theta) : \NormBoundedExpr{\nu}{\theta}{\mathscr D_{\mathbf{B}(0,1)}(\mathbf R^n,Y)}  \} = 0.\]
 	\end{enumerate}
 	Here, by convention a polynomial function $P : \mathbf R^n \to Y$ of degree $-1$ is the zero function.
 \end{definition}

Suppose $i$ is a nonnegative integer.
In \cite{2021},
 Menne defines the notion of 
 $\boldsymbol{\nu}_{\mathbf B(0,1)}^i$ pointwise differential 
 by replacing $\nu$ with 
 $\boldsymbol{\nu}_{\mathbf B(0,1)}^i$ in \ref{IntroDef:Pt}, 
 and proves a criterion for differentiability
  and four types of results: 
  Borel regularity of the differential, 
  a Rademacher-Stepanov-type theorem, rectifiability of the family of $k$ jets, 
  and a Lusin-type approximation by functions of class $k$. 

In this paper, we extends those statement from \cite{2021} 
from the case of $\boldsymbol{\nu}_{\mathbf B(0,1)}^i$ pointwise differentials to 
the case of $\nu_{i,p}$ pointwise differentials.
 There is an exception in formulation of the Borel regularity theorem, 
 instead of treating the pointwise differential of a fixed distribution 
 as a function of point in euclidean space, 
 we prove the Borel regularity of the $\nu$ pointwise differential 
 as function in variable of point in Euclidean space and point in the space of distributions.
Now, we state the main results of the paper.
\paragraph{A Criteria of differentiability}
A criterion for the differentiability of distributions is proved in 
which the $\nu_{i,p}$ pointwise differentiability of a distribution can
 be deduced from the $\nu_{i,p}$ pointwise differentiability of its derivatives.

\begin{introTheorem}[see \ref{ThmB}]\label{theorem-2}
	Suppose $i,m$ are nonnegative integers, $k$ is an integer, $0<\alpha \leq 1$, $Y$ is a Banach space, $1\leq p \leq  \infty$, $T \in \mathscr D'(\mathbf R^n,Y)$, and $a\in \mathbf R^n$.
	\begin{enumerate}
		\item If  $k\geq 0$ and 
		for every $\xi\in \Xi(n,m)$ the distribution $\D^\xi T$ 
		is $\nu_{i,p}$ pointwise differentiable of order $k$ at $a$, 
		then for every $l=0,1,\dots,m-1$ and $\beta \in \Xi(n,l)$ 
		the distribution $\D^\beta T$ is $\nu_{i,p}$ pointwise differentiable of order $k-l+m$ at $a$.
		\item If $k+\alpha \geq 0$ and for every $\xi\in \Xi(n,m)$
		 the distribution $\D^\xi T$ is $\nu_{i,p}$ pointwise differentiable of order $(k,\alpha)$ at $a$, then for every $l=0,1,\dots,m-1$ and $\beta \in \Xi(n,l)$ the distribution  $\D^\beta T$ is $\nu_{i,p}$ pointwise differentiable of order $k-l+m$ at $a$.
	\end{enumerate}
\end{introTheorem}

In \citep{2021},
 this type of result is obtained as a consequence of 
 a Poincaré type inequality \citep[3.11]{2021}. 
 We use a different approach, following the remark \cite[3.12]{2021}. 
 Specifically, we first reduce to the case $k=1$ by induction
  and prove \ref{EFF} which 
  \textit{relates the deformation of $T$ at a point and their partial derivatives}.
  This relation allow us to conclude that the differentiablility of $\D_j T$ for $j\in\{1,\dots,n\}$ 
  implies the deformation of $T$ is a Cauchy filter in the space of distributions
  clustering a distribution whose first order deriviatve vanishes. Theorem
  \ref{theorem-2} follows from estimating the convergence behavior of the deformation. 
  We also deduce Re\v{s}hetnyak's theorem \citep{MR0136849} as a corollary (see \ref{Rešhetnyak’s}).
	
\paragraph{Borel regularity of the differential}
 In \cite{2021}, the regularity of the pointwise differential is studied as a function mapping points in Euclidean space into the space of symmetric functions. In this paper, we study the regularity of the $\nu$ pointwise differential as a function mapping points and distributions into the space of symmetric functions. To state our results we use Bourbaki's notation in \BourbakiTVS{} in which the $\mathscr D'(\mathbf R^n,Y)_b$ is the space of distributions endowed with uniform convergence on bounded sets in $\mathscr D(\mathbf R^n,Y)$ (see appendix \ref{loc} for full detail).

\begin{introTheorem}\emph{(see \ref{Borel-function-of-ptD})}
	\label{introThmA}
	Suppose $Y$ is a Banach space,
	 $\kappa\in \mathbf R$, 
	 $k$ is an  integer,
	 and $\nu$ is a continuous seminorm on $\mathscr D_{\mathbf B(0,1)}(\mathbf R^n,Y)$. Let $Y_k$ be vector space of polynomial function of degree at most $k$ with values in $Y'$(here we use the convention that for every integer $k< 0$, the vector space $ Y_k$ is zero). Then 
	\begin{enumerate}
		\item   The set $E$ of  $(a,T,P) \in \mathbf R^n\times  \mathscr E'(\mathbf R^n,Y) \times Y_k$ satisfying 
		\[\limsup_{r\to 0^+}r^{\kappa}\sup \{ (T-P)^{a,r}(\theta) : \NormBoundedExpr{\nu}{\theta}{\mathscr D_{\mathbf B(0,1)}(\mathbf R^n,Y)} \} < + \infty\] 
		is a countable union of closed subsets of $\mathbf R^n\times  \mathscr E'(\mathbf R^n,Y)_b \times Y_k$. Moreover, if $Y$ is finite dimensional, then the conclusion can be improved to the set $E$ is a countable union of compact subsets of $\mathbf R^n\times  \mathscr E'(\mathbf R^n,Y)_b \times Y_k$. 
		\item  The set $E_0$ of  $(a,T,P) \in \mathbf R^n\times  \mathscr E'(\mathbf R^n,Y) \times Y_k$ satisfying 
		\[\lim\limits_{r\to 0^+}r^{\kappa}\sup \{ (T-P)^{a,r}(\theta) : \NormBoundedExpr{\nu}{\theta}{\mathscr D_{\mathbf B(0,1)}(\mathbf R^n,Y)} \} =0\] 
		is a Borel subset of $\mathbf R^n\times  \mathscr E'(\mathbf R^n,Y)_b \times Y_k$.
		\item  The set $D$ of  $(a,T,P) \in \mathbf R^n\times  \mathscr D'(\mathbf R^n,Y) \times Y_k$ satisfying 
		\[\limsup_{r\to 0^+}r^{\kappa}\sup \{ (T-P)^{a,r}(\theta) : \NormBoundedExpr{\nu}{\theta}{\mathscr D_{\mathbf B(0,1)}(\mathbf R^n,Y)} \} < + \infty\] 
		is a countable union of closed subsets of $\mathbf R^n\times  \mathscr D'(\mathbf R^n,Y)_b \times Y_k$. 
		\item  The set $D_0$ of  $(a,T,P) \in \mathbf R^n\times  \mathscr D'(\mathbf R^n,Y) \times Y_k$ satisfying 
		\[\lim\limits_{r\to 0^+}r^{\kappa}\sup \{ (T-P)^{a,r}(\theta) : \NormBoundedExpr{\nu}{\theta}{\mathscr D_{\mathbf B(0,1)}(\mathbf R^n,Y)} \} =0\] 
		is a Borel subset of $\mathbf R^n\times  \mathscr D'(\mathbf R^n,Y)_b \times Y_k$.
	\end{enumerate}
\end{introTheorem}      

The first step is to reduced the problem to the case of 
space of compactly supported distributions 
(see \ref{Borel-function-of-ptD:enum:1}). 
On the set $\mathbf R^n \times \mathscr E'(\mathbf R^n,Y)\times Y_k$, 
we analyze the definition of the $\nu$ pointwise differentiability of order $k$ 
or $(k,\alpha)$ to reduce the problem to showing that, 
for a fixed $\theta\in \mathscr D(\mathbf R^n,Y)$, 
the function mapping $(a,T)$ onto $f_\theta(T) = T_x(\theta(r^{-1}(x-a)))$ 
is continuous when $\mathscr D'(\mathbf R^n,Y)$ is equipped with
 the topology of uniform convergence on bounded sets
  (see \ref{lower-semicontinuous-lemma}).
   In \ref{E_c-strictly-stronger-then-E_s}, 
   we show that the function $f_\theta$ is in general not continuous, 
   when $\mathscr D'(\mathbf R^n)$ is equipped with the weak topology.
    As a corollary, we intersects the graph of the $\nu$ pointwise differential 
	with $\mathbf R^n \times \{ T\} \times Y_k$ to obtain
	 the Borel regularity formulated in \cite{2021}.

\begin{corollary}\label{cor:Borel reg}(See \ref{ThmA})
	Suppose $Y$ is a Banach space, $k$ is an integer, 
	$0<\alpha  \leq  1$,
	$k+\alpha \geq 0$,
	$\nu$ is a continuous seminorm on
	 $\mathscr D_{\mathbf B(0,1)}(\mathbf R^n,Y)$,
	and $T \in \mathscr  D'(\mathbf R^n,Y)$. 
	Let $Y_k$ be the vector space of polynomial function 
	of degree at most $k$ with values in $Y'$.
	Then the following two statements hold
	\begin{enumerate}[ref=\thetheorem(\theenumi)]
		\item The set $D'$ of all $(a,P)\in  \mathbf R^n\times Y_k$ such that $T$ is $\nu$ pointwise differentiable of order $(k,\alpha)$ with $k$ jet $P$ at $a$ is a countable union of compact subsets of $\mathbf R^n\times Y_k$.
		\item If $k\geq 0$, then the set $D'_0$ of all $(a,P)\in  \mathbf R^n\times Y_k$ such that $T$ is $\nu$ pointwise differentiable of order $k$ with $k$ jet $P$ at $a$ is a Borel subset of $\mathbf R^n$.
		\item  The set of $a \in \mathbf R^n$ such that $T$ is $\nu$ pointwise differentiable of order $(k,\alpha)$ at $a$ is a countable union of compact subsets of $\mathbf R^n$.
		\item If $k\geq 0$ and $Y$ is separable, 
		then $D'_0$ is a Borel function. Furthermore, $\dmn D'_0$ is a Borel set.
		\item If $k\geq 0$ and $Y$ is finite dimensional, 
		then for every $j\in \{0,1,\dots,k\}$ 
		the set of all $(a,T,P) \in \mathbf R^n \times \mathscr D'(\mathbf R^n,Y) \times Y_k$ 
		such that $T$ is $\nu$ pointwise differentiable of order $k$ at $a$ is a Borel set and the function mapping $(a,T) \in X$ onto $\pt \D^j T(a)$ is a Borel function on $\mathbf R^n \times \mathscr D'(\mathbf R^n,Y)_s$ whose domain is a Borel set.
	\end{enumerate}
\end{corollary}

We note that if $Y$ is finite dimensional normed space, then the Borel sets of $\mathscr D'(\mathbf R^n,Y)_s$ and $\mathscr D'(\mathbf R^n,Y)_b$ agree, this result can be deduced from \ref{Lusin theorem}, \ref{Bounded subset of strong topology of space of distribution}, and \cite[Theorem 2.2]{862d4889d5ff49cf87cff80c14f3769a}.

 \paragraph{Rectifiability of the family of $k$ jets}
 Using the first two main results  (Theorem \ref{theorem-2} and Theorem \ref{introThmA}), 
 we deduce the rectifiability theorem.
 \begin{introTheorem}\emph{(see \ref{Rect})}\label{Thm-C}
 	Suppose $Y$ is a Banach space, $i,k$ are nonnegative integers, $0<\alpha \leq 1$, $0 \leq k + \alpha$, $1\leq p<+\infty$, $T \in \mathscr D'(\mathbf R^n,Y)$, and $A$ is the set of points $a\in \mathbf R^n$ such that $T$ is $\nu_{i,p}$ pointwise differentiable of order $(k,\alpha)$ at $a$. Then there exists a sequence of functions $f_1,f_2\dots,$ of class $(k,\alpha)$ mapping $\mathbf R^n$ into $Y'$ and compact sets $C_1,C_2,\dots,$ such that $A = \bigcup_{i=1}^\infty C_i$ and
 	$\pt \D^m T(a) = \D^m f_j(a)$
 		for every positive integer $j$, $a\in C_j$ and $m\in \{ 0,\dots,k\}$.
 \end{introTheorem}
\paragraph{Lusin approximation by functions of class $k$}
In view of the Borel regularity for pointwise differential (Theorem \ref{introThmA}),
 the Lusin approximation by function of class $k$ for pointwise differentials \cite[4.25]{2021}, 
 and the fact that $\nu$ pointwise differentiability implies pointwise differentiability, 
 we have the following Lusin type approximation theorem.
\begin{introTheorem}\label{introThmE}\emph{(See \ref{Lusin theorem})}
	Suppose $Y$ is a Banach space, $Y'$ is separable, $i,k$ are nonnegative integers, $\nu$ is a real valued seminnorm on $\mathscr D_{\mathbf B(0,1)}(\mathbf R^n,Y)$, and $A$ is the set of $a\in \mathbf R^n$ such that the distribution $T$ is $\nu$ pointwise differentiable of order $k$. Then for every $0<\epsilon<+\infty$, there exists
	a function $g$ of class $k$ mapping $\mathbf R^n$ into $Y'$ such that
	\[ \mathscr L^n(A \sim \{ a : \pt \D^m T(a) = \D^m g(a)\quad \text{for $m=0,\dots,k$}\}< \epsilon.\] 
\end{introTheorem}
 \paragraph{A Rademacher-Stepanov type theorem}
The Rademacher type theorem for $\nu_{i,p}$ pointwise differentials has occured in 
 \citep[P. 189, Theorem 10]{MR0136849} with $i=0,k=0$, and $1<p<\infty$.
Mennee proved the case $i=1,k=0$, and $1<p<+\infty$ in \citep{MR3023856},
and proved the case $i\geq 0$, $k\geq 0$, and $p=\infty$ in \citep{2021}.
 
In this paper, we prove the Rademacher type theorem for $\nu_{i,p}$ pointwise differentials with
$1<p<\infty,i>1$, and $k\geq 0$, by using the observation by Menne in \citep{MR3023856} that 
this type of results with $1<p<\infty$ depends only on a \textit{priori} estimates of certain elliptic operators.
But, because we involve the case $k>0$, we need the higher order differentiability results in \cite{Agmon2011} for 
weak solution of poly-Laplacian. We notice that the method used for the case in \cite{2021} make no use of 
$\mathbf L_p$ theory of PDE, and the case $p=1$ is not solved yet. 
Now, we state the Rademacher type theorem, i.e. the case $1<p<\infty$, $i\geq 0$, and $k\geq 0$.

\begin{introTheorem}\label{introThmD}(See \ref{Radamacher theorem})
	Suppose $n$ is a positive integer, $i,k$ are nonnegative integers, $0<\alpha \leq 1$, $T \in \mathscr D'(\mathbf R^n)$, and $A$ is the set of all $a\in \mathbf R^n$ such that $T$ is $\nu_{i,p}$ pointwise differentiable of order $(k-1,1)$ at $a$. Then $T$ is $\nu_{i,p}$ pointwise differentiable of order $k$ at $\mathscr L^n$ almost all $a\in A$.
\end{introTheorem}

\def\SECTION{{Section} }
 \def\APPENDIX{{Appendix} }
\paragraph{Organization of the paper}
Section \ref{sec:Notation} contains the notation which will be used in this paper,
 section \ref{Appendix:WF} contains the Poincar\'e  inequality for vector valued weakly differentiable functions 
 and interpolation inequality for the later reference in estimate,
 section \ref{Def_Pt} contains the proof of deformation estimate (\ref{Thm}) 
 and theorem \ref{theorem-2}. 
 Theorem \ref{introThmA} is proved in section \ref{Grph_Pt}
  which relies upon several general topological results concerning uniform spaces and locally convex spaces. 
  In the interest of completeness, 
  basic definitions and results for uniform spaces are collected in appendix \ref{uni}. 
  The main statement needed is a criterion for continuity (see \ref{lower-semicontinuous-lemma}) 
  and a projection result (see \ref{rmk:Suslin projection} and \ref{Suslin-projection}). 
  Appendix \ref{loc} collects statements on locally convex spaces 
  which are necessary to formulate the topological properties of distributions and test functions.
   Theorem \ref{Thm-C} and Theorem \ref{introThmE} are contained in section \ref{sec:Lusin plus Rect}.
    Theorem \ref{introThmD} is contained in section \ref{Rade_Pt}
	 which relies upon the \textit{a priori} estimates for Sobolev functions and the poly-Laplacian in 
	  section \ref{est}.
 
 In view of the connection \ref{Modulus formulation for Lp pointwise differentiability} between the
  $\nu_{0,p}$ pointwise differentiability at a point of order $k$ 
 and the differentiability of function in Lebesgue spaces in \cite{MR0136849},
anologous to \cite[Theorem 2]{MR0136849},
 it is natural to ask the following questions:
let $m,i$ be nonnegative integers, $\Lap^i u = T \in \mathscr D'(\mathbf R^n)$,
\begin{itemize}[label={}]
	\item If $0<\alpha<1$ and $k$ is an integer with $k\geq -1$, and $T$ is $\nu_{m,p}$ pointwise 
	 differentiable of order $(k,\alpha)$ at $x_0$, 
	 does it follows that for every multi-index $\alpha$ with $ |\xi|\leq 2i$, $\D^\xi u$ is $\nu_{m,p}$ pointwise differentiable of order
	  $(k+2i-|\xi|,\alpha)$ at $x_0$ ?
	\item If $k$ is a nonnegative integer, 
	$T$ is $\nu_{m,p}$ pointwise differentiable of order $k$ at
	 a set $A$ of positive $\mathscr L^n$ measure, does it follows that
	  for every multi-index $\xi$ with $|\xi|\leq 2i$, $\D^\xi u$ is $\nu_{m,p}$ pointwise differentiable 
	 of order $k+2i-|\xi|$ at $\mathscr L^n$ almost every points in $A$ ?
\end{itemize}

The following convention for the notion of vector space will be used in this paper.
The base field of vector space and algebra will be the real numbers unless otherwise stated.
\section{Acknowledgement}
The author is grateful to 
  Professor Ulrich Menne 
  for providing this topic,
  pointing out the generalization \ref{Borel-function-of-ptD} 
  of Borel regularity result in \cite{2021} , 
  and the discussions of ideas of the proofs, to
  Sean McCurdy
  for the suggestion on the 
  mathematical writting which greatly improved the readability of the paper and the discussion of the paper \cite{Agmon2011},
   which is a key step in the proof of Rademacher theorem \ref{Radamacher theorem}.
   
This research was supported
	through the National Science and Technology Council through grant nos.
	MOST 108-2115-M-003-016-MY3, MOST 110-2115-M-003-017, MOST
	111-2115-M-003-014, and NSTC 112-2115-M-003-001.
\section{Notation}\label{sec:Notation}
Some of the notations is defined in the introduction section,
we should repeat it for completeness.
A relation $f$ is a subset of $X \times Y$ for some sets $X,Y$.
When $f$ is a relation, we also write $xfy$ for $(x,y)\in f$.
The domain $\dmn f$ of $f$ is the set of $x\in X$ such that there exists $y\in Y$
such that $xfy$. The \textit{inverse} $f^{-1}$ of $f$ is the relation
$Y\times X \cap \{ (y,x) : xfy\}$. Composition $g\circ f$ of two relations
$f \subset X \times Y$ and $g\subset Y \times Z$ is defined to be the set of points $(x,z) \in X \times Z$
such that there exist $y\in Y$ with $xfy$ and $ygz$.
A relation $f$ is a function if and only if for every $y,y'\in Y$ with $xfy$ and $xfy'$, $y=y'$.
Let $\D$ denote the total derivative and $\D_v$ the partial derivative at direction $v\in \mathbf R^n$.
When $v$ is the $j$th standard basis vector, we write $\D_j$ instead of $\D_{e_j}$.
Whenever $\mu$ is a measure over a set $X$ and $A \subset X$, let 
$\phi\weight A (T) = \phi(A \cap T)$ for every $T \subset X$.

\begin{defparagraph}[Lebesgue space]\label{defpar:Lebesgue space}
	Suppose $Y$ is a Banach space, $1\leq p <+\infty$, and $\mu$ is a Radon measure on an open set $U$ of $\mathbf R^n$. Let $\mathbf A(\mu,Y)$ denote the set of all $\mu$ measurable $Y$ valued functions $f$ on $U$ with such that there exist a separable closed subspace $Z$ of $Y$ such that $f$ maps $\mu$ almost every point of $U$ into $Z$.
	The set of all  $f\in \mathbf A(\mu,Y)$ with
	\[   \mu_p (f) = \begin{cases}
			(\int |f|^p \ d \mu)^{1/p}  : 1\leq p<+\infty \\
			\inf \{ t : \mu (\{ x : |f(x)|>t\}) =0 : p = +\infty\}
		\end{cases}<+\infty,\]
	is denoted by $\mathbf L_p(\mu,Y)$.
\end{defparagraph}
\begin{defparagraph}[Space of smooth function]\label{defpar:Space of smooth function}
Suppose $Y$ is a Banach space, $U$ is an open subset of $\mathbf R^n$, $0\leq k\in \mathbf Z$, and $0 \leq \alpha\leq 1$. Let $\mathscr E (U,Y)$ denote the vector space of all smooth $Y$ valued maps defined on $U$. Whenever $K$ is a compact subset  of $U$, we define $\mathscr D_K(U, Y) = \mathscr E (U,Y) \cap \{ f : \spt f \subset K\}$ and $\mathscr D(U, Y) = \bigcup \{ \mathscr D_K(U , Y) : K \text{ compact subset of } U \}$. For each $0\leq i \in \mathbf Z$ and compact subset $K$ of $U$ we define for every $\theta \in \mathscr E(U, Y)$,
	$\boldsymbol{\nu}_{K}^i(\theta) = \sup \{ \| \D^i \theta (x) \| : x \in K \}$. We equip $\mathscr E(U,Y)$ with the complete metrizable locally convex topology defined by the set of real valued seminorms $\boldsymbol{\nu}_{K}^i$, where $i$ and $K$ run through the nonnegative integers and compact subsets $K$ of $U$, respectively.
\end{defparagraph}
\begin{defparagraph}[Distribution]\label{defpar:Distribution}
	Suppose $U$ is an open subset of $ \mathbf R^n$. A \textit{distribution} of $Y$ in $U$ is a linear map $T : \mathscr D(U,Y) \to \mathbf R$ whose restriction to $\mathscr D_K(U,Y)$, corresponding to every  compact subset $K$ of $U$, is continuous for the complete metrisable locally convex topology induced by the sequence of norms $\boldsymbol{\nu}^0_K, \boldsymbol{\nu}^1_K,  \boldsymbol{\nu}^2_K,\cdots$.

	The locally convex topology on $\mathscr D(U,Y)$ will be the strict inductive limit of
	$ \{ \mathscr D_K(U,Y) : K\subset U \text{ compact} \}$
	with respect to the inclusion. The set of distributions of $Y$ in $U$ is denoted by $\mathscr D'(U,Y)$.
	For every $k$ times differentiable function $\theta : U \to Y$ and $\alpha \in \Xi(n,k)$ define $\D^\alpha \theta = \langle e^\alpha , \D^k \theta \rangle $.

	Suppose $Y$ is a Banach space,
	 $n$ is a positive integer, 
	 $i$ is a nonnegative integer, 
	 $k$ is an integer,
	  $0<\alpha \leq 1$, $0\leq k+\alpha$,
	   $1<p<+\infty$, 
	   and $T\in \mathscr D'(U,Y)$.
	Given an open subset $U$ of $\mathbf R^n$ 
	and a $\mathscr L^n$ locally integrable function $f : U \to Y'$ on $U$. 
	For every $\theta \in \mathscr D(U,Y)$, 
	we write  $f_x(\theta(x)) = \int \langle \theta(x) , f(x) \rangle d\mathscr L^n_x$.
	For every $\theta \in \mathscr D(U,Y)$, we write $T(\theta) = T_x(\theta(x))$ to indicate the integration variable. Finally, for every $a\in \mathbf R^n$ and $0<r<\infty$,
	we define
	\begin{equation*}
		T^{a,r}(\theta) = r^{-n}T_x(\theta(r^{-1}(x-a))) \quad \text{for every $\theta \in \mathscr D(\mathbf R^n,Y)$}.
	\end{equation*}
\end{defparagraph}

\begin{defparagraph}\label{Lp Modulus for Distribution}
	Suppose $n$ is a positive integer, $U$ is an open subset of $\mathbf R^n$, $Y$ is a Banach space,  $T$ is a linear function mapping $\mathscr D(U,Y)$ into $\mathbf R$, $S \subset U$, $0\geq i\in \mathbf Z$ and either
	\begin{enumerate}
		\item $1<p,q<+\infty$ and $p^{-1}+q^{-1}=1$,
		\item or $p = 1$ and $q=+\infty$,
		\item or $p = +\infty$ and $q=1$,
	\end{enumerate}
	we define
	\[
		\textstyle |T|_{i,p;S} = \begin{cases}
			\sup \{ T(\theta) : \theta\in \mathscr D(U,Y), \spt \theta \subset S, \mathscr L^n_q(\D^{-i} \theta)\leq 1\} & : i \leq 0 \\
			\sum_{0\leq |
			\alpha| \leq i} |\D^\alpha T|_{0,p;S}                                                                        & : i>0.
		\end{cases}
	\]
	For every  locally $\mathscr L^n$ integrable function $f : U \to   Y'$, let $T$ be the distribution associate with $f$, we write $|f|_{i,p;S} = |T|_{i,p;S}$. In case $S=\mathbf U(a,r)$, $a\in \mathbf R^n,0<r<+\infty$, we use the abbreviation,
	$|T|_{i,p;a,r} = |T|_{i,p;S}$.
	In case $i = 0$, we write $|T|_{p;S}$ and $|T|_{p;a,r}$ instead of $|T|_{0,p;S}$ and $|T|_{0,p;a,r}$, respectively.
\end{defparagraph}

\begin{defparagraph}[Weakly differentiable function]\label{defpar:Weakly differentiable function}
	Suppose $k$ is a nonnegative integer, $U$ is an open subset of $\mathbf R^n$, and $Y$ is a Banach space. Let $e_1,\dots,e_n $ be the standard basis of $\mathbf R^n$ with corresponding dual basis $\omega_1,\dots,\omega_n$.
	A function $f : U \to Y$ is weakly differentiable of order $k$ if and only if for every multi index $\alpha$ with $j = |\alpha| \leq k$ there exist a $\mathscr L^n$ locally integrable function $g_\alpha$ mapping $U$ into $Y$ such that the following integration by part identities hold
	\begin{equation}\label{Integration by part for weakly differentiable function 1}
		\begin{aligned}
			\textstyle \int \langle f(x),y' \rangle \cdot \D^\alpha \theta(x) \ d\mathscr L^n_x = (-1)^j \int \langle g_\alpha(x),y' \rangle \cdot \theta(x)\ d\mathscr L^n_x \\
			\text{for every } \theta \in \mathscr D(U),y'\in Y'.
		\end{aligned}
	\end{equation}
	The function $g_\alpha$ is determined by \eqref{Integration by part for weakly differentiable function 1} uniquely $\mathscr L^n$ almost everywhere, and we write  $\D^\alpha f = g_\alpha$  and $\D^kf(x) = \sum_{\alpha \in \Xi(n,k)} \omega^\alpha \cdot \D^\alpha f$, where $\omega^\alpha = \omega_{1}^{\alpha(1)}\cdots \omega_n^{\alpha(n)}$.
	The differential  $\D^if(x) = \sum_{\alpha \in \Xi(n,i)} \omega_\alpha \cdot \D^\alpha f(x) \in \bigodot^i(\mathbf R^n,Y)$ is defined  for $\mathscr L^n$ almost all $x\in U$.
If $Y$ is separable, then we have
	\begin{gather*}
		\textstyle (\D^\alpha f)( x) = \lim\limits_{r\to 0^+} \frac{1}{\mathscr L^n (\mathbf B(x,r))}\int_{\mathbf B(x,r)} \D^\alpha f d\mathscr L^n\\
		\text{for $\mathscr L^n$ almost all $x\in U$}.
	\end{gather*}
	For every $1\leq p \leq +\infty$,
	the space of $k$ times weakly differentiable functions with $p$ summable $k$ th order  differential over $U$ is denoted by  $W_{k,p}(U,Y)$. The locally convex topology on $W_{k,p}(U,Y)$ is defined by the seminorm
	that maps every $f\in W_{k,p}(U,Y)$ on to $\sum_{i=0}^k (\int \|\D^i f\|^p d\mathscr L^n)^{1/p}$. 
	The closure of the subspace of $W_{k,p}(U,Y)$ consisting $f : U \to Y$ such that $\int_{U\sim K} |f| d\mathscr L^n = 0$ for some  compact subset $K$ of $U$ is denoted by $W_{k,p}^\diamond(U,Y)$.
The completion of $\mathscr D(U,Y)$ by the norm $\nu_{k,p}$ is $W_{k,p}^\diamond(U,Y)$. In case $Y=\mathbf R$, let  $W_{k,p}^\diamond(U) =W_{k,p}^\diamond(U,\mathbf R) ,W_{k,p}(U) = W_{k,p}(U,\mathbf R)$ and $W_{k,p}^{loc}(U)$ be the set of $\mathbf R$ valued $\mathscr L^n$ measurable function on $U$ such that for every $\chi \in \mathscr D(U)$ , the function $\chi \cdot f$ belongs to $W_{k,p}(U)$.
\end{defparagraph}

\paragraph{Symmetric algebra}
\begin{defparagraph}\label{defpar:Symmetric algebra}
	Suppose $E,Y$ are vector spaces and $k$ is a  nonnegative integer.
	 A $k$ linear map $f : \prod_{i=1}^k E  \to Y$ is \textit{symmetric} 
	 if and only if 
	\[ f(v_1,\dots,v_k) = f(v_{\sigma(1)},\dots,v_{\sigma(k)}) 
	\quad \text{for every permutation $\sigma$ on $\{ 1,\dots,k\}$ }.\]
	The vector space of all symmetric $k$ linear maps of $E$ into $Y$ 
	is denoted by $\bigodot^k (E, Y )$\label{symmetri tensor}, and $\bigodot^\ast(E,Y)$
	their direct sum over nonnegative integers $i$.
	The \textit{symmetric algebra} over $E$ is denoted by $\bigodot_\ast E$,
	 $\odot$  denote the multiplication of  symmetric algebra.
	  For every nonnegative integer $n$, $e_1,\dots,e_n, v \in E$, 
	  and function $\alpha$ mapping $ \{1,\dots,n\}$ into $\mathbf N \cup \{ 0\}$, we write
	\begin{gather*}
		\odot_{i=1}^n v = v^n , \quad 
		e^\alpha = e_1^{\alpha(1)} \odot \dots \odot e_n^{\alpha(n)}.
	\end{gather*}

\cite[1.9.4]{MR41:1976}
	The diagonal map $\Upsilon$ is the unique unit-preserving graded algebra homomorphism mapping $\bigodot_\ast E$ into $\bigodot_\ast E \otimes \bigodot_\ast E$ and characterized by
	\[ \Upsilon(x) = 1\otimes x + x\otimes 1 \quad \text{for every $x\in E$}.\]

\cite[1.10.5]{MR41:1976} 
Whenever $X,Y,Z$ are vector spaces and $h$ is a linear function mapping $X \otimes Y$ into $Z$,  we define the linear function $\odot$ mapping $\bigodot^\ast(E,X) \otimes \bigodot^\ast(E,Y)$ into $\bigodot^\ast(E,Z)$, characterized by
	\[\textstyle \phi \odot \psi = h \circ (\phi \otimes \psi) \circ \Upsilon\quad \text{for every $(\phi,\psi) \in\bigodot^\ast(E,X) \times \bigodot^\ast(E,Y) $}.\]
	 Given a nonnegative integer $m$,
	a map $P : \mathbf R^n \to Y$ is called a homogeneous polynomial map of degree $m$ if  and only if it is of the form, 
	\[
	P(x) = \langle  {x^m / m!} , \psi \rangle \text{ for  $x\in \mathbf R^n$},
	\]
	where $\psi \in \Hom(\bigodot_m \mathbf R^n,Y)$. In this case we denote $\deg(P) = m$.
	
	 A function $f : \mathbf R^n \to Y$ is called a \textit{polynomial map} or \textit{polynomial function} if and only if $f$ is a finite sum of homogeneous polynomial functions $P_1,\dots,P_k$, and the degree of $f$ is defined to be the $\deg f = \sup \{ \deg P_i : i = 1,\dots,k\}$. Given a basis $e_1,\dots,e_n$  of $\mathbf R^n$, for every $x\in \mathbf R^n$ we write $ x =\sum_{j=1}^n x_i e_i$ to compute
\[ f(x) = f(\sum_{j=1}^n x_i e_i) = \sum_{i,\alpha} x_1^{\alpha_1}\cdots x_n^{\alpha_n}
 \left\langle {\frac{e^\alpha}{\alpha !}}, P_i\right\rangle,\]
where the  summation is taken over all $i$ belonging to $\{ 0,\dots,k\}$, and then all $(\alpha_1,\dots,\alpha_n) \in (\mathbf N\cup \{0\})^n$ such that the sum of $\alpha_1,\dots,\alpha_n$ equals $i$.
The set of all $\alpha : \{ 1,\dots,n\} \to \mathbf N \cup \{ 0\}$ with $
\sum_{i=1}^n\alpha_i = k
$ is denoted by $\Xi(n,k)$.

Suppose $n$ is a postive integer, $p,q$ are nonnegative integers, $Y$ is a seminormed space with seminorm $\nu$, $\psi \in \bigodot^p(\mathbf R^n,Y), \phi \in \bigodot^q(\mathbf R^n,Y)$, and let
\[ \|\psi \| = \sup \{ \nu\circ \psi(v_1,\dots,v_p) : |v_i| \leq 1 , v_i\in \mathbf R^n ,i=1,\dots,p\}.\]
Then 
$\| \phi \odot \psi \| \leq \binom{p+q}{p} \| \phi \| \cdot \| \psi \|$(see \cite[1.10.5]{MR41:1976}).
Suppose $E,F$ are two inner product spaces. 
 The inner product on $E_1 \otimes E_2$ is given as in \cite[1.7.9]{MR41:1976},
 \[(\xi_1\otimes \xi_2,\xi_2\otimes \eta_2)  \mapsto \xi_1 \bullet \xi_2 \cdot \eta_1\bullet \eta_2 \]
 for $\xi_1,\xi_2 \in E$ and $\eta_1,\eta_2 \in  F$.
 . When $Y$ is a inner product space, the inner product of $\bigodot_k(\mathbf R^n,Y)$ is given as in \cite[1.10.6]{MR41:1976},
 induced by $\bigodot_i\mathbf R^n \otimes Y$ via  the isomorphisms
 \[\textstyle    \bigodot_k\mathbf R^n \otimes Y \simeq \Hom(\bigodot_k\mathbf R^n,Y) \simeq   \bigodot^k(\mathbf R^n,Y). \]

\end{defparagraph}

\section{Properties of weakly differentiable functions}\label{Appendix:WF}

\begin{lemma}\label{Poincare ineqaulity for zero boundary value functio}Suppose $1\leq p<\infty$, $Y$ is a Banach space, and $f : \mathbf R^n \to Y$ is function of class $1$ with compact support. Then for every $0< r <\infty$, and $b\in \mathbf R^n$ such that $\spt f \subset \mathbf  U(b,r)$,
	\[ \textstyle \left(\int |f|^p d\mathscr L^n\right)^{1/p} \leq  r   \left(\int |\D f(x)|^pd\mathscr L^n_x\right)^{1/p}.\] 
\end{lemma}
\begin{proof}
	For every $\alpha \in Y'$, we apply \cite[7.14]{MR1814364} with $u$ replaced by $\alpha \circ f$ to obtain the estimate at every $x\in\mathbf  U(b,r)$,
	\begin{equation*}
		\textstyle |f(x)| \leq   {\frac{1}{n\boldsymbol{\alpha}(n)}} \int \frac{|\D f(y)| }{|x-y|^{n-1}} \ d \mathscr L^n_y,
	\end{equation*}
	and use \cite[7.12]{MR1814364} to estimate $g = |\D f|$,
	\begin{equation*}
		\textstyle |V_{1/n}g|_p \leq n \boldsymbol{\alpha}(n)^{1-1/n} \mathscr L^n(\mathbf  U(b,r))^{1/n} |\D f|_p.,
	\end{equation*}
	where $V_{1/n}g(x) = \int \frac{g(y)}{|x-y|^n-1} \ d \mathscr L^n_y$.
	
\end{proof}

\begin{theorem}\label{Poincare inequality}Suppose $k$ is a positive integer, $Y$ is a Banach space, $1\leq p <+\infty$,
	 $0<r<+\infty$, $a\in \mathbf R^n$, and $f\in W_{k,p}(\mathbf U(a,r),Y)$.
	 Then there exists a polynomial function $P$ mapping $\mathbf R^n$ into $Y$ of degree at most $k-1$ such that
	\begin{align*}
		\textstyle \left( \int_{\mathbf U(a,r)}|\D^i(f-P)|^p \ d \mathscr L^n\right)^{1/p} \leq  (2^n r )^{k-i} \left( \int_{\mathbf U(a,r)}|\D^kf|^p \ d \mathscr L^n\right)^{1/p} \\
		\text{for every $i \in \{ 0,1,\dots,k-1\}$ }.
	\end{align*}
\end{theorem}

\begin{proof}
	For the case $k=1$,
	define $\phi = \mathscr L^n(\mathbf U(a,r))^{-1}\int_{\mathbf U(a,r)} f \ d \mathscr L^n$.
	For each $\alpha \in Y'$, we
	apply \cite[Lemma 7.16]{MR1814364} with $u$ replaced by $\alpha \circ f$ to obtain for $\mathscr L^n$
	 almost every $x \in \mathbf U(a,r)$,
	\[ |f(x) - f_{\mathbf U(a,r)}| \leq
	\textstyle \frac{2^n}{n\boldsymbol{\alpha}(n)} \int |x-y|^{-n+1}|\D f(y)| \ d \mathscr L^n_y, \]
	where $f_{\mathbf U(a,r)} = \mathscr L^n(\mathbf U(a,r))^{-1}\int_{\mathbf U(a,r)} f \ d \mathscr L^n$.
	We apply \cite[Lemma 7.12]{MR1814364} with $p,q,\mu,\delta,f$ replaced by 
	$p,p,1/n,0,|\D f|$ to estimate potential of $|\D f|$, 
	\[ \textstyle  |V_{1/n}(|\D f|)|_p \leq n \boldsymbol{\alpha}(n) r \left( \int |\D f|^p \ d \mathscr L^n \right)^{1/p},\]
	this proves the case $k=1$. Proceeding inductively, 
	we apply \ref{Poincare inequality} with $f,k,Y$ replaced by $\D^{k-1} f,1,\bigodot^i(\mathbf R^n,Y)$ to obtain a  $\phi \in \bigodot^i(\mathbf R^n,Y)$ with
	\begin{equation*}
		\textstyle \left( \int_{\mathbf U(a,r)}|\D^{k-1}f - \phi|^p \ d \mathscr L^n\right)^{1/p}  \leq  2^n r   \left( \int_{\mathbf U(a,r)}|\D^kf|^p \ d \mathscr L^n\right)^{1/p}.
	\end{equation*}
	We define $Q(x) = \langle x^{(k-1)}/(k-1)!,\phi \rangle$ for every $x\in \mathbf R^n$ and apply $\ref{Poincare inequality}$ with $f,k,Y$ replaced by $f-Q,k-1,Y$ to obtain a polynomial function $H$ on $\mathbf R^n$ with values in $Y$ of degree at most $k-2$ with
	\begin{align*} 
		\textstyle \left( \int_{\mathbf U(a,r)}  | \D^i (f- Q-H) |^p \   d \mathscr L^n \right)^{1/p}  
		\leq (2^n r)^{k-1-i} \left( \int_{\mathbf U(a,r)}|\D^{k-1} (f - Q)|^p \ d \mathscr L^n \right)^{1/p} 
	\end{align*}
	for every $i\in \{ 0,1,\dots,k-2\}$.
	Noticing that $\D^{k-1} Q = \phi$, the conclusion follows with $P = H+Q$.
\end{proof}

\begin{propparagraph}\emph{\cite[(7.44)]{MR1814364}}\label{Duality of Poincare inequality}
	Suppose $n,k$ are positive integers, $Y$ is a Banach space, $a\in \mathbf R^n$ $0<r<+\infty$, and $1\leq p<+\infty$.
	Then 
	inductively applying \ref{Poincare ineqaulity for zero boundary value functio},
	 we obtain for every $f\in \mathscr D(\mathbf U(a,r),Y)$ and $j \in \{ 0,1,\dots,k\}$,
	\[  \textstyle \left(\int |\D^j f|^p d\mathscr L^n\right)^{1/p} \leq r^{k-j} \textstyle \left(\int |\D^kf|^p d\mathscr L^n\right)^{1/p}.\]
	In particular, $|\D^k
	\cdot |_{p}$ and $\sum_{j=0}^k |\D^j \cdot|_{p}$ are equivalent norms on $\mathscr D(U,Y)$. 
	Dually, for every nonnegative integer $k$ and $\alpha \in \Xi(n,j)$, we have
	\[ |\D^\alpha T|_{-k,p;a,r} \leq r^{k-j}|T|_{p;a,r}.  \]
\end{propparagraph}
\begin{theorem}\label{append:Poincare inequality for zero boundary value function on balls}Suppose $n,k$ are positive integers. Then for every $a\in \mathbf R^n$, $0<r<+\infty$, integer $0\leq j\leq k$, and $u \in W_{k,p}^\diamond(\mathbf U(a,r))$,
	\begin{equation*}
		|\D^j u|_{p;a,r} \leq  r^{k-j}|\D^k u |_{p;a,r}.
	\end{equation*}
\end{theorem}
\begin{proof}
	Since $\mathscr D(\mathbf U(a,r))$ is dense in $W_{k,p}^\diamond(\mathbf U(a,r))$, we may assume $u \in \mathscr D(\mathbf U(a,r))$. In this case, the assertion follows from \ref{Duality of Poincare inequality}.
\end{proof}
\begin{theorem}\label{append:Interpolation inequality for Balls}Suppose $0<\epsilon<+\infty$, $1\leq p<+\infty$ and $n,k$ are positive integers. Then there exists a constant $C$ depending on $n,k,p,\epsilon$ with the following property. The inequality 
	\begin{equation}
		r^{i}|\D^i u|_{p;a,r} \leq \epsilon r^{k}|\D^k u|_{p;a,r} + C |u|_{p;a,r}
	\end{equation}
	holds for every $a\in \mathbf R^n$, $0<r<+\infty$, $u\in W_{k,p}(\mathbf U(a,r))$, and integer $0<i<k$.
\end{theorem}

\begin{proof}
	It is sufficient to prove the case $a = 0$ and $r = 1$. If the assertion were false for some $0<\epsilon<+\infty$, 
	then there exists a integer $0<i<k$ 
	and a sequence $u_1,u_2,\dots \in W_{k,p}(\mathbf U(0,1))$ such that for every $j \in \{ 1,2,\dots\}$,
	\[ 1 = |\D^i u_j|_{p;0,1} > \epsilon |\D^k u_j|_{p;0,1} + j |u_j|_{p;0,1}. \]
	By Rellich compactness theorem \cite[Theorem 7.26]{MR1814364}, we may assume there exists $u \in W_{k,p}(\mathbf U(0,1))$ such that 
	\[ \lim\limits_{j \to \infty} |u-u_j|_{k-1,p;0,1} = 0,\]
	but $\lim\limits_{j\to \infty}|u_j|_{p;0,1} = 0$, we must have $u=0$, this contradicts to $|\D^i u|_{p;0,1} =1$.
\end{proof}

\section{A criterion for pointwise differentiability}\label{Def_Pt}

The main result of this section is Theorem \ref{theorem-2}(\ref{ThmB}). As mentioned in the introduction, the proof of Theorem \ref{theorem-2} depends upon a deformation result (\ref{Thm}). Before stating the deformation result, we state a basic results about the relationship between pointwise differentiability under different norms. In particular, \ref{pt order k in  0 sense implies pt order k  in -i sense} will be useful in the proof of Theorem \ref{introThmD}. As a corollary of Theorem \ref{theorem-2}, we also obtain a Resetnyak-type Theorem, which will also be useful in the proof of Theorem \ref{introThmD} (see \ref{Pointwise comparison decay to global comparison decay}).
Using the notation introduced in \ref{Lp Modulus for Distribution}, we may restate the definition of $\nu_{i,p}$ differentiability as follows.
\begin{propparagraph}\label{Modulus formulation for Lp pointwise differentiability}
	Suppose $i$ is a nonnegative integer, $1\leq p <+\infty$, the seminorm $\nu$ is $\nu_{i,p}$, and either $p=1,q=+\infty$ or $p=+\infty,q =1$  or $1<p<+\infty,p^{-1}+q^{-1} = 1$. Then the distribution $T \in \mathscr D'(\mathbf R^n,Y)$ is $\nu_{i,p}$ pointwise differentiable of order $k$, $(k,\alpha)$ if and only if there exist a polynomial function $P$ of degree at most $k$ mapping $\mathbf R^n$ into $Y'$ such that
	\begin{equation*}
		\begin{aligned}
			&\lim\limits_{r\to 0^+} r^{-n/q-k-i} |T-P|_{-i,q;a,r} = 0,\\
			&(resp. \quad \limsup\limits_{r\to 0^+} r^{-n/q-k-\alpha-i} |T-P|_{-i,q;a,r} <+\infty.
		\end{aligned}
	\end{equation*} 
\end{propparagraph}
The $\nu_{i,p}$ pointwise differential is a stronger differentiability than pointwise differentiability. In fact, 
we have the following two comparison of differentiability results \ref{comparison of pointwise differentiability:1} and
\ref{comparison of pointwise differentiability:2}.
\begin{propparagraph}\label{comparison of pointwise differentiability:1}
	Suppose $\nu$ is a real valued seminorm on $\mathscr D_{\mathbf {B}(0,1)}(\mathbf R^n,Y)$, $k$ is  an integer, $0 < \alpha \leq 1$, $k+\alpha \geq 0$, and $a\in \mathbf R^n$. Then $T$  is $\nu$ pointwise differentiable of order $k$ (of order $(k,\alpha)$) at $a$ implies $T$ is pointwise differentiable of order $k$ (of order $(k,\alpha)$) at $a$.
\end{propparagraph}

\begin{propparagraph}\label{comparison of pointwise differentiability:2}
	Suppose $\nu,\mu$ are two seminorms on $\mathscr D_{\mathbf B(0,1)}(\mathbf R^n,Y)$, $k$ is  an integer, $0 < \alpha \leq 1$, $k+\alpha \geq 0$, and $\nu \leq M \mu$ for some $0\leq M<+\infty $. Then 
	whenever the distribution $T \in \mathscr D'(\mathbf R^n,Y)$ is $\nu$ pointwise differentiable of order $(k,\alpha)$ (order $k$) at $a$, it is also $\mu$ pointwise differentiable of order $(k,\alpha)$ (order $k$) at $a$.
\end{propparagraph}
By Poincar\'e's inequality \ref{Duality of Poincare inequality}, we see that $\nu_{0,p}$ and $\nu_{i,p}$ are comparable. Combining this observation with the the implication of differentiability \ref{comparison of pointwise differentiability:2}, we obtain the following statement.
\begin{propparagraph}\label{pt order k in  0 sense implies pt order k  in -i sense}
	Suppose $i,k$ are nonnegative integers,  $0<\alpha \leq 1$, $0 \leq k + \alpha$, $1\leq p<+\infty$, and $a\in \mathbf R^n$. Then 
	\begin{enumerate}
		\item If the distribution $T$ is $\nu_{0,p}$ pointwise differentiable of order $0\leq k$ at $a$, then the distribution $T$ is $\nu_{i,p}$ pointwise differentiable of order $k$ at $a$.
		\item If the distribution $T$ is $\nu_{0,p}$ pointwise differentiable of order $(k,\alpha)$ at $a$, then the distribution $T$ is $\nu_{i,p}$ pointwise differentiable of order $(k,\alpha)$ at $a$.
	\end{enumerate}
\end{propparagraph}

We now state a constancy theorem which will be important in the proof (see \eqref{Existence:deformation}) of the deformation lemma.
\begin{propparagraph}[Constancy theorem]\cite[4.1.4]{MR41:1976}\label{Constancy theorem}
	Suppose $Y$ is a Banach space, $n$ is a positive integer, $T \in \mathscr D'(\mathbf R^n,Y)$ and  $\D_jT = 0$ for $ j=1,\dots,n$. 
	Then there exist a unique $\alpha \in Y'$ such that
	\[ \textstyle T(\theta) = \int \langle \theta, \alpha \rangle \ d\mathscr L^n \quad \text{for every $\theta \in \mathscr D(\mathbf R^n,Y)$} .\]
\end{propparagraph}
\opt{Description:1}{
Now, we estimate the deformation of distribution by their partial derivatives. Given $n,a,Y,\theta,T,r,s$, and $X_j$ as in  \ref{Thm}, 
we compute the  derivative of the function $f(r) =r^{-n} \theta(r^{-1}(x-a))$ to obtain \ref{claim in lemma:Relation between deformation and partial derivative}
\[ f'(r) = \sum_{i=1}^n r^{-n} \D_i(X_i\theta)(r^{-1}(x-a)),\]
and apply the distribution $T$ to both side of equation, we infer from a basic analytic fact the equation \ref{lemma:Relation between deformation and partial derivative}
\begin{equation*}
	r^{-n} \theta(r^{-1}(x-a)) - s^{-n} \theta(s^{-1}(x-a)) = \int_s^r \sum_{i=1}^n t^{-n} \D_i(X_i\theta)(t^{-1}(x-a))d\mathscr  L^1_t.
\end{equation*}
Under the additional hypothesis on $\alpha,C,\nu$, we may use the bounds (\ref{Estimate:partial der})
on partial derivative together with the equation just obtained to estimate \ref{Estimate:Deformation}
\[ (T^{a,r} - T^{a,s})(\theta) \leq C \cdot M n(\alpha + 1)^{-1} (r^{\alpha+1} -s^{\alpha +1}).\]
 Finally, a further application of constancy theorem (\ref{Constancy theorem}) we obtain the existence of the limit of deformation \ref{Existence:deformation}.}
\begin{lemma}\label{Thm}\label{EFF}
	Suppose $n$ is a positive integer,
	 $a \in \mathbf R^n$, $Y$ is a Banach space, 
	$\theta \in \mathscr D(\mathbf R^n,Y) , T \in \mathscr D'(\mathbf R^n,Y)$, $0<s\leq r  <\infty$, and for $j=1,\dots,n$, $ X_j: \mathbf R^n \to \mathbf R$ is the projection onto the $j$th coordinate. Then 
	\begin{align}\label{lemma:Relation between deformation and partial derivative}
		(T^{a,r} - T^{a,s})(\theta) 
		= \sum_{j=1}^{n}\int_{s}^r t^{-n}(\D_jT)_x([X_j\theta](t^{-1}(x-a))) d\mathscr L^1_t.
	\end{align}
	Moreover, if $\alpha>-1 $, $C$ is nonnegative real number, $\nu$ is  a real valued semimorm on $\mathscr D_{\boldsymbol{B}(0,1)}(\mathbf R^n,Y)$ such that for every $0<t \leq \delta$, $j=1,\dots,n$,
	\begin{equation}\label{eq:cond:localness of the norm}
		\begin{aligned}
		\sup \{ \nu(X_j\theta) : \theta \in \mathscr D_{\mathbf B(0,1)}(\mathbf R^n) , \nu(\theta) \leq 1\}  \leq C, \\
		t^{-\alpha} \sup \{(\D_j T)^{a,t}(\theta): \NormBoundedExpr{\nu}{\theta}{\mathscr D_{\mathbf B(0,1)}}(\mathbf R^n,Y) \} \leq M<+\infty,
	\end{aligned}
	\end{equation}
	 then the following two assertions  hold.
	\begin{enumerate}
		\item For $0<s<r\leq \delta$,
		\begin{equation}\label{Estimate:Deformation}
			\begin{aligned}
			\sup \{(T^{a,r}- T^{a,s})(\theta) : \NormBoundedExpr{\nu}{\theta}{\mathscr D_{\mathbf B(0,1)}}(\mathbf R^n,Y) \}\\
			\leq C M n(\alpha + 1)^{-1} (r^{\alpha+1} -s^{\alpha +1}).
		\end{aligned}
		\end{equation}
		\item There exist a $P : \mathbf R^n \to Y'$ polynomial map of degree $0$ such that,
		\begin{equation}\label{Existence:deformation} 
			\begin{aligned}
		\sup\{ r^{-\alpha-1}(T^{a,r}-P)(\theta) : 0<r\leq \delta,\NormBoundedExpr{\nu}{\theta}{\mathscr D_{\mathbf B(0,1)}}(\mathbf R^n,Y)  \} \\
		\leq C M n(\alpha + 1)^{-1}
	\end{aligned}
		\end{equation}
	
	\end{enumerate}
\end{lemma}
\begin{proof}
	To prove the first equation \eqref{lemma:Relation between deformation and partial derivative},
	 we compute for each $x\in \mathbf R^n$  the derivative of  the function $f(r) =r^{-n} \theta(r^{-1}(x-a))$ whenever $ 0<r<\infty$,
	\begin{equation}\label{claim in lemma:Relation between deformation and partial derivative}
		\begin{aligned}
		f'(r) &= -n r^{-n-1}\theta(r^{-1}(x-a)) - r^{-n-2}\langle x-a , \D\theta(r^{-1}(x-a)) \rangle \\
		      &= \sum_{i=1}^{n} -r^{-n-1}\theta(r^{-1}(x-a)) - r^{-n-2}X_i(x-a)\D_i\theta(r^{-1}(x-a)) \\
		      &=\sum_{i=1}^n r^{-n} \D_i(X_i\theta)(r^{-1}(x-a)),
	\end{aligned}
	\end{equation}
	integrating $f$ to obtain,
\begin{equation*}
	\textstyle r^{-n} \theta(r^{-1}(x-a)) - s^{-n} \theta(s^{-1}(x-a)) = \int_s^r \sum_{i=1}^n t^{-n} \D_i(X_i\theta)(t^{-1}(x-a))d\mathscr  L^1_t.
\end{equation*}
	Applying $T$ to both sides, and observing that $T,\int$ can be exchanged, gives \eqref{lemma:Relation between deformation and partial derivative}.
	
	Under the additional hypothesis on $\alpha,C,\nu$,
	 for every $\theta \in \mathscr D(\mathbf R^n,Y),\nu(\theta) \leq 1 $, $j\in \{ 1,\dots,n\}$
	  and $0<t \leq \delta$, we have
	\begin{equation}\label{Estimate:partial der}
		| t^{-n}(\D_jT)_x([X_j\theta](t^{-1}(x-a))) |\leq C M t^\alpha.
	\end{equation}
	Thus we estimate for every $0<s \leq r \leq \delta$,
	\begin{equation}\label{Estimate of deformation by partial derivative}
		\begin{aligned}
		(T^{a,r} - T^{a,s})(\theta) & \textstyle \leq \sum_{j=1}^{n} \int_s^r | t^{-n}(\D_jT)_x([X_j\theta](t^{-1}(x-a))) | d\mathscr L^1_t\\
							& \textstyle \leq n \int_s^r  C M t^\alpha d\mathscr L^1_t\\
							&= C \cdot M n(\alpha + 1)^{-1} (r^{\alpha+1} -s^{\alpha +1}).
					\end{aligned}
	\end{equation}
	
	 Since $\nu$ is real valued and $\alpha + 1 > 0$, $T^{a,r}$ converges to a $P\in \mathscr  D'(\mathbf R^n,Y)$ with respect to $\sigma(\mathscr D(\mathbf R^n,Y),
	\mathscr D'(\mathbf R^n,Y))$ as $r\to 0^+$, in view of \eqref{eq:cond:localness of the norm}, the first order partial derivative of $P$ vanishes. We infer from the constancy theorem \ref{Constancy theorem} that the polynomial map $P$ must be of degree $0$.
	 
	 For every $0<s,r \leq \delta$, $ \theta \in \mathscr D_{\mathbf B(0,1)}(\mathbf R^n,Y) $, and  $\nu(\theta)\leq 1 $ we estimate
	\begin{align}\label{Convergent of the deformation to a polynomial function}
		\begin{aligned}
	r^{-\alpha-1}(T^{a,r}-P)(\theta) = r^{-1-\alpha} (T^{a,r}-T^{a,s})(\theta) + r^{-1-\alpha}(T^{a,s}-P)(\theta) \\
	\leq C \cdot M n(\alpha + 1)^{-1} (1 -(s/r)^{\alpha +1}) + r^{-1-\alpha} (T^{a,s}-P)(\theta).
\end{aligned}
	\end{align}
	Let $s\to 0^+$ to conclude that
	\begin{align*} 
	\sup\{ r^{-\alpha-1}(T^{a,r}-P)(\theta) : 0<r\leq \delta,\NormBoundedExpr{\nu}{\theta}{\mathscr D_{\mathbf B(0,1)}}(\mathbf R^n,Y)  \} \\
	\leq C M n(\alpha + 1)^{-1}.\end{align*}
\end{proof}
\begin{propparagraph}\label{example-L_p}
	Suppose $Y$ is a Banach space, $1\leq p \leq \infty$, and $i$ is a positive integer.
	Then the norm $\nu_{i,p}$ on $\mathscr D_{\mathbf B(0,1)}(\mathbf R^n,Y)$ satisfies the first inequality of \eqref{eq:cond:localness of the norm}. In fact, if $X_j$ is as in \eqref{eq:cond:localness of the norm}, then for $\theta \in \mathscr D_{\mathbf B(0,1)}(\mathbf R^n,Y)$, we use Leibniz formula to estimate
	\[ \nu_{i,p}(X_j\theta) \leq \sum_{k=0}^i \nu_{0,p}(\D^k\theta \odot \D^{i-k} X_j)  
	\leq \nu_{i,p}(\theta) + \nu_{i-1,p}(\theta). \]
	The first inequality of \eqref{eq:cond:localness of the norm} also holds for other norms. For example,
assume $1<p<\infty$ and $i$ is a positive integer. The norm $\nu$ defined by
	\[ \textstyle \nu(\theta) =\nu_{i,p}(\Lap^i \theta) \quad \text{for $\theta \in \mathscr D_{\mathbf B(0,1)}(\mathbf R^n) $}\]
	 satisfies the conditions in \eqref{eq:cond:localness of the norm}. This follows from \textit{a priori} estimates on $\Lap^i$ (see \ref{Prior estimate of Dirichlet problem on balls for zero boundary value function}). 
\end{propparagraph}

\begin{theorem}\label{Cor-Poincare}\label{ThmB} \ref{theorem-2}
	Suppose $n$ is a positive integer, 
	$i,m$ are nonnegative integers, 
	$k$ is an integer, 
	$Y$ is a Banach space, 
	$1\leq p \leq \infty$, $0<\alpha \leq 1$,
	$T \in \mathscr D'(\mathbf R^n,Y)$,
	 and $a\in \mathbf R^n$.
	\begin{enumerate}
		\item \label{Cor-Poincare-1}If  $k\geq 0$ and for every $\xi\in \Xi(n,m)$ the distribution $\D^\xi T$ is $\nu_{i,p}$ pointwise differentiable of order $k$ at $a$, then for every $l=0,1,\dots,m-1$ and $\beta \in \Xi(n,l)$ the distribution $\D^\beta T$ is $\nu_{i,p}$ pointwise differentiable of order $k-l+m$ at $a$.
		\item \label{Cor-Poincare-2}If $k+\alpha \geq 0$ and for every $\xi\in \Xi(n,m)$ 
		the distribution $\D^\xi T$ is 
		$\nu_{i,p}$ pointwise differentiable of order of order $(k,\alpha)$ at $a$,
		 then for every $l=0,1,\dots,m-1$ and $\beta \in \Xi(n,l)$ 
		 the distribution  $\D^\beta T$ is 
		 $\nu_{i,p}$ pointwise differentiable of order $k-l+m$ at $a$.
	\end{enumerate}
\end{theorem}

\begin{proof}
	We will first prove \ref{Cor-Poincare-2}.
	Using induction on $m$, it is sufficient to prove $m=1$.
	According to \ref{example-L_p}, we have a constant $C$ as in \ref{Thm},
	 for every $i =1,\dots,n$, let $P_i$ be the $k$ jet of $\D_iT$ at $a$.
	  If $\D_i T$ is $\nu_{i,p}$ pointwise differentiable at $a$ of order $k$, 
	  then by \cite[Theorem 3.11]{2021} 
	  the distribution $T$ is pointwise differentiable of order $k+1$ at $a$.
	  Thus the $k+1$ jet $Q$ of $T$ at $a$ satisfies $\D_i Q = P_i , i=1,\dots,n$.
From the hypothesis we have, whenever $\epsilon,\delta>0$ and $i\in \{ 1,\dots,n\}$, 
	\[ \sup \{ t^{-\alpha-k} (\D_iT - \D_i Q)^{t,a}(\theta): \NormBoundedExpr{\nu_{i,p}}{\theta}{\mathscr D_{\mathbf B(0,1)}(\mathbf R^n,Y)}, 0 < t \leq \delta\} \leq \epsilon, \]
	taking account of \ref{example-L_p}, 
	we may apply \ref{Thm} with $(T,\alpha,\nu,M)$ 
	replaced by $(T-Q,\alpha+k,\nu_{i,p},\epsilon)$ 
	to obtain a polynomial map $H : \mathbf R^n \to Y'$ degree at most $0$ such that,
	\begin{align*}
		\sup \{ t^{-\alpha-k-1} (T -Q-H)^{t,a}(\theta): \NormBoundedExpr{\nu_{i,p}}{\theta}{\mathscr D_{\mathbf B(0,1)}(\mathbf R^n,Y)} 0 < t \leq \delta\} \\
		\leq C \epsilon n(\alpha + 1)^{-1} .
	\end{align*} 
	Since $H$ do not depends on $\epsilon$, this proves the case $0<\alpha \leq 1$. 
	Now we prove the first statement \ref{Cor-Poincare-1}, in this case $\alpha = 0$. 
	Apply the above argument to each $\epsilon$,
and let $\epsilon \downarrow 0$ to obtain the conclusion.
\end{proof}

\begin{definition}\label{def:Lebesgue point}
	Suppose $f$ is a locally $\mathscr L^n$ integrable function on $\mathbf R^n$. Given $1\leq p \leq +\infty$, a point $a\in \mathbf R^n$ is called $p$th order Lebesgue point of $f$ if and only if 
	\[ \lim\limits_{t\to 0^+}r^{-n/p}(\mathscr L^n \weight \mathbf B(a,r))_{(p)}(f-f(a)) = 0.\]
\end{definition}

\begin{corollary}[Rešhetnyak's theorem]\label{Rešhetnyak’s}
	Suppose $n$ is a positive integer, $i$ is a nonnegative integer, $U$ is an open subset of $\mathbf R^n$, and one of the following three cases hold
	\begin{enumerate}
		 \item $p = 1$,  $q=+\infty$.
		\item $1<p,q<+\infty$, $p^{-1}+q^{-1} = 1$.
		\item $p= +\infty$, $q =1$.
	\end{enumerate}
 Then for every $f \in W_{i,p}^{loc}(U)$, $a$ is a $p$ th order Lebesgue point of $\D^if$, and distribution
	$T \in \mathscr D'(\mathbf R^n)$ with $T|\mathscr D(U) = f$ the distibution $T$ is $\nu_{0,q}$ pointwise differentiable of order $i$ at $a$ and the $i$ jet at $a$ is $\sum_{j=0}^i\langle  (x-a)^j/j!,\D^jf(a)\rangle$.
\end{corollary}
\begin{proof}
	Assume $U = \mathbf R^n$ and  $T = f \in W_{i,p}(\mathbf R^n)$, $\alpha \in \Xi(n,i)$, $0<r<+\infty$, and notice that the number $|\D^\alpha f - \D^\alpha f(a)|_{p;a,r}$ equals
	$
			(\mathscr L^n\weight \mathbf B(a,r))_{(p)}(f-f(a))
		$.
	By \ref{Modulus formulation for Lp pointwise differentiability} and \ref{def:Lebesgue point},
	 the distribution $\D^\alpha f$ is $\nu_{0,q}$ pointwise differentiable of order $0$ at $a$.
	By \ref{Cor-Poincare}, the distribution $f$ is $\nu_{0,q}$ pointwise differentiable of order $i$ and the conclusion follows 
	from \citep[2.12]{2021}.
\end{proof}

\section{The graph of a pointwise differential}\label{Grph_Pt}
In this section, we prove Theorem \ref{introThmA} (see \ref{Borel-function-of-ptD}). The key ingredients are a continuity criterion \ref{lower-semicontinuous-lemma} and a related lower semi-continuity criterion \ref{lower-semicontinuous-lemma}. In \ref{E_c-strictly-stronger-then-E_s}, we observe that the topology of uniform convergence and weak topology are different and that the functions in \ref{lower-semicontinuous-lemma} are not in general continuous.

This section uses some specialized notation from topology, the theory of locally convex spaces and uniform spaces. For example, $E_B'$ denotes the dual space of a locally convex space $E$ equipped with uniform convergence on every member of $B$, in particular $\mathscr D'(\mathbf R^n,Y)_b$ denotes the space of distribution equipped with topology of uniform convergence on every bounded subset of $\mathscr D(\mathbf R^n,Y)$,
 $S^\circ$ denotes the polar of $S$. Full definitions and basis properties are provided in \ref{loc} and \ref{uni}.

\begin{lemma}\label{equi-criteria-special-case}
	Suppose $X$ is a locally compact Hausdorff space, $E$ is a locally convex space,  $\mathbf B$ is a family of bounded subsets of $E$, and $u $ is a function mapping $X\times E_{\mathbf B}'$ into $\mathbf R$ with the following two properties.
	\begin{enumerate}
		\item For every $b \in E_{\mathbf B}'$ and $x\in X$, $u(\cdot , b)$ is continuous and $u(x,\cdot) \in E'^\ast$.
		\item For every compact subset $K$ of $X$ there exists $S\in \mathbf B$ such that
		\[ \sup \{ u(x, T ) : T \in S^\circ , x\in K\} < +\infty. \]
	\end{enumerate}
	Then $u$ is continuous.
\end{lemma}

\begin{proof}
	Since $X$ is locally compact there exists a family $R$ of compact subsets of $X$ such that the family $\{ \Int S : S\in R\} $ covers $X$. We claim that for every compact subset $K$ of $X$,
	\[ \{u(x,\cdot ) : x\in K\} \cap E'^\ast\]
	is equicontinuous with respect to the space $E'$ equipped with the topology of uniform convergence on $\mathbf B$.
	
	To verify the claim, note that for every $S\in \mathbf B$ the polar (see \ref{Duality}) $S^{\circ}$ of $S$ 
	is a neighborhood of $0$ with respect to the topology of uniform convergence on $\mathbf B$ on $E'$.
	By hypothesis, for each $K\in R$ there exists $S\in \mathbf B$ such that
	\[ \sup \{ u(x, T ) : T \in S^\circ , x\in K\} < +\infty. \]
	By \ref{equi}, the set of $\{ u(x, \cdot  ) :  x\in K\}$  of functions is equicontinuous. The conclusion of the lemma follows from applying \ref{continuity of fibration} with $X,Y,Z$  replace by $ \mathbf R,E_\mathbf{B}',\mathbf R$.
\end{proof}

\begin{remark}\label{equi-prod-continuity}
	Suppose $X$ is a locally compact Hausdorff space, $E$ is a locally convex space, and $u : X \times E'_{\mathbf B} \to \mathbf R$ is continuous. Then for every compact subset $K$ of $X$ we have
	\[ \sup \{ u(x, T ) : T \in S^\circ , x\in K\} < +\infty. \]
	This follows from the definition of $E'_\mathbf{B}$ (see  \ref{Duality}).
\end{remark}

\begin{lemma}\label{lower-semicontinuous-lemma}
	Suppose $Y$ is a Banach space. Then for every $\theta \in \mathscr D(\mathbf R^n,Y)$, the function $f_\theta$ mapping $(a,T) \in \mathbf R^n \times \mathscr D'(\mathbf R^n,Y)_b$ onto $T_x(\theta(x-a))$  is continuous. In particular, for every $S\subset \mathscr D(\mathbf R^n,Y)$ the map 
	 \begin{gather*}
	 	(a,T) \mapsto \sup \{ t^\kappa \cdot T_x(\theta(t^{-1}(x-a))) : 0< t \leq r , \theta \in S\}\\
	 	 \text{for $(a,T)\in \mathbf R^n\times \mathscr D'(\mathbf R^n,Y)$}
	 \end{gather*}
	is lower semicontinuous on $\mathbf R^n \times \mathscr D'(\mathbf R^n,Y)_b$.
\end{lemma}

\begin{proof} 
	
	For every $a\in \mathbf R^n$, the function 
	$T \mapsto T_x(\theta(x-a))$ of $\mathscr D'(\mathbf R^n,Y)$
	 is continuous {(see \ref{Def:Stong dual and bidual})}. For every compact subset $K$, let $C = K+\spt \theta$. Notice $C$ is compact and the set
	$S = \{  \theta(\cdot -a) : a \in C \} $ is bounded in $\mathscr D(\mathbf R^n,Y)$. By \ref{equi} with $\mathbf B$ replaced by the class of bounded subset of $\mathscr D(\mathbf R^n,Y)$ 
	we conclude that 
	$ \langle \theta(x-a),(\cdot)_x \rangle : a \in C\}$   
	 is an equicontinuous subset of $\mathscr D'(\mathbf R^n,Y)_b$.
	 We apply \ref{family-of-Hyperplane-bounded-over-nbd} with $E$, $F$, and $U$ replaced by $\mathscr D'(\mathbf R^n,Y)_b$, $\mathbf R$, and $\mathbf R \cap \{ t :|t| \leq 1\}$ to obtain a bounded subset $S$ of $\mathscr D(\mathbf R^n,Y)$ such that
	\[ \sup \{ \langle \theta(x-a),T_x \rangle : a\in K, T \in S^\circ  \} <+\infty. \]
	The conclusion follows from  \ref{equi-criteria-special-case} with $u,E$, and $\mathbf B$ replaced by $f,\mathscr D(\mathbf R^n,Y)$, and the class of bounded subset of $\mathscr D(\mathbf R^n,Y)$.
\end{proof}
The necessity of passing to the topology of uniform convergence on bounded subsets on $\mathscr D'(\mathbf R^n,Y)$ is shown by the following example.
\begin{example}\label{E_c-strictly-stronger-then-E_s}
	Suppose $ n \in \mathbf N$ is a positive integer and  $Y$ is a Banach space with $\dim Y > 0$. Then the topology of $\mathscr D'(\mathbf R^n,Y)$  induced by uniform convergence on compact subsets of $\mathscr D(\mathbf R^n,Y)$ is strictly finer than the weak topology of $\mathscr D'(\mathbf R^n,Y)$ induced by $\mathscr D(\mathbf R^n,Y)$. 
	
	If they were the same topology, then \ref{compar-topology} every compact subset of $\mathscr D(\mathbf R^n,Y)$ would be contained in a convex envelope of some nonempty finite subset of $\mathscr D(\mathbf R^n,Y)$. That is, every compact set would have finite affine dimension. Towards a contradiction , let 
	\begin{gather*}
		\Phi \in \mathscr D_{\mathbf B(0,1)}(\mathbf R^n),\quad 0\neq v \in \mathbf R^n, \quad 0\neq \xi \in Y,\quad  \text{and} \\
		\Phi(x)\neq 0 \quad \text{for every $x\in \mathbf U(0,1)$}.
	\end{gather*}We may construct a sequence of functions $g_1,g_2,\dots,\in\mathscr D(\mathbf R^n,Y)$ by
	\[ g_i(x) = \Phi(x-a_i)\xi \quad \text{ for  $x \in \mathbf R^n$  },  \]
	where $a_i=i^{-1}v$ and $i \in \mathbf N$, which forms a linearly independent compact subset of  $\mathscr D(\mathbf R^n,Y)$. In fact, the $g_i$ is nonzero every where on $\mathbf{U}(a_i,1)$ for every $i \in \mathbf N$, and if $J \subset \mathbf N$  is finite and $j\in J$, there exist a point $x\in \mathbf{U}(a_j,1) \sim \bigcup\{ \mathbf{B}(a_i,1) : i\in J\sim \{ j\}\}$ such that $g_j(x)\neq 0$. 
	
	The function $f_\theta$ defined in the proof of \ref{lower-semicontinuous-lemma} with $\theta$ replaced by $\Phi\cdot \xi$ is not continuous with respect to $\sigma(\mathscr D'(\mathbf R^n,Y),\mathscr D(\mathbf R^n,Y))$ (see \ref{Def:Weak topology}). In fact, if $f_{\Phi\cdot \xi}$ were continuous with respect to $\sigma(\mathscr D'(\mathbf R^n,Y),\mathscr D(\mathbf R^n,Y))$, then we apply \ref{equi-prod-continuity} with $X$ and $\mathbf B$  replaced by $\mathbf R^n$ and the set of finite subsets of $\mathscr D(\mathbf R^n,Y)$ to obtain a finite set $S$ of $\mathscr D(\mathbf R^n,Y)$ such that
	\[ \sup \{ T(g_i) : T \in S^\circ,i \in \mathbf N\} = 1 <+\infty.\]
	We apply \BourbakiTVS{\BourbakiCiteNumber{2}{6}{3}. Theorem 1} with $F$,$G$, $\sigma(F,G)$ and $M$ replaced by $\mathscr D(\mathbf R^n,Y)$, $\mathscr D'(\mathbf R^n,Y)$, $\sigma(\mathscr D(\mathbf R^n,Y),\mathscr D'(\mathbf R^n,Y))$ and $S$, to conclude that the sequence is contained in the closed convex envelope of the finite set $S \cup \{0\}$, which contradicts to the second paragraph.
\end{example}

\Copy{ThmA}{
\begin{theorem}\label{Borel-function-of-ptD}\ref{introThmA}
	Suppose $n$ is a positive integer, $Y$ is a Banach space, $\kappa\in \mathbf R$, and $\nu$ is a continuous seminorm on $\mathscr D_{\mathbf B(0,1)}(\mathbf R^n,Y)$. Let $Y_k$ be vector space of polynomial functions of degree at most $k$ with values in $Y'$(here we use the convention that for every integer $k< 0$, the vector space $ Y_k$ is zero). Then 
	\begin{enumerate}[ref=\thetheorem(\theenumi)]
		\item \label{Borel-function-of-ptD:enum:1}  The set $E$ of  $(a,T,P) \in \mathbf R^n\times  \mathscr E'(\mathbf R^n,Y) \times Y_k$ satisfying 
		\[\limsup_{r\to 0^+}r^{\kappa}\sup \{ (T-P)^{a,r}(\theta) : \NormBoundedExpr{\nu}{\theta}{\mathscr D_{\mathbf B(0,1)}(\mathbf R^n,Y)} \} < + \infty\] 
		is a countable union of closed subsets of $\mathbf R^n\times  \mathscr E'(\mathbf R^n,Y)_b \times Y_k$. Moreover, if $Y$ is finite dimensional, then the conclusion can be improved to be the set $E$ is a countable union of compact subsets of $\mathbf R^n\times  \mathscr E'(\mathbf R^n,Y)_b \times Y_k$. 
		\item \label{Borel-function-of-ptD:enum:2} The set $E_0$ of  $(a,T,P) \in \mathbf R^n\times  \mathscr E'(\mathbf R^n,Y) \times Y_k$ satisfying 
		\[\lim\limits_{r\to 0^+}r^{\kappa}\sup \{ (T-P)^{a,r}(\theta) : \NormBoundedExpr{\nu}{\theta}{\mathscr D_{\mathbf B(0,1)}(\mathbf R^n,Y)} \} =0\] 
		is a Borel subset of $\mathbf R^n\times  \mathscr E'(\mathbf R^n,Y)_b \times Y_k$.
		\item \label{Borel-function-of-ptD:enum:3} The set $D$ of  $(a,T,P) \in \mathbf R^n\times  \mathscr D'(\mathbf R^n,Y) \times Y_k$ satisfying 
		\[\limsup_{r\to 0^+}r^{\kappa}\sup \{ (T-P)^{a,r}(\theta) : \NormBoundedExpr{\nu}{\theta}{\mathscr D_{\mathbf B(0,1)}(\mathbf R^n,Y)} \} < + \infty\] 
		is a countable union of closed subsets of $\mathbf R^n\times  \mathscr D'(\mathbf R^n,Y)_b \times Y_k$. 
		\item \label{Borel-function-of-ptD:enum:4} The set $D_0$ of  $(a,T,P) \in \mathbf R^n\times  \mathscr D'(\mathbf R^n,Y) \times Y_k$ satisfying 
		\[\lim\limits_{r\to 0^+}r^{\kappa}\sup \{ (T-P)^{a,r}(\theta) : \NormBoundedExpr{\nu}{\theta}{\mathscr D_{\mathbf B(0,1)}(\mathbf R^n,Y)} \} =0\] 
		is a Borel subset of $\mathbf R^n\times  \mathscr D'(\mathbf R^n,Y)_b \times Y_k$.
	\end{enumerate}
\end{theorem}
}
\begin{proof}
	We will first prove \ref{Borel-function-of-ptD:enum:1} and \ref{Borel-function-of-ptD:enum:2}, which will imply \ref{Borel-function-of-ptD:enum:3} and \ref{Borel-function-of-ptD:enum:4}, respectively.
 
	Choose a norm $q$ to define the unique topology on $Y_k$ which is compatible with its vector space structure.
	For every positive integer $i,j$ and $0<\epsilon<+\infty$ we define
	\begin{equation}
		\begin{aligned}
			 A(i,j) &= \mathscr E'(\mathbf R^n,Y) \cap \{ T: 
			 	     |T(\theta)| \leq j \boldsymbol{\nu}_{\mathbf B(0,i)}^i(\theta) \quad \text{for every $\theta \in \mathscr E(\mathbf R^n,Y)$}.\}, \\
			B(i) & = Y_k \cap \{ P : q(P) \leq i\}, \qquad C(i,j)= \mathbf B(0,i) \times A(i,j)\times B_i,
		\end{aligned}
	\end{equation}
and $E(i,\epsilon)$ to be the set of $(a,T,P) \in \mathbf R^n\times  \mathscr E'(\mathbf R^n,Y) \times Y_k$ such that
\[ |(T-P)^{a,r}_x(\theta)| \leq \nu(\theta) r^{n+\kappa} \epsilon\]
for every $\theta \in \mathscr D_{\mathbf B(0,1)}( \mathbf R^n,Y)$ and $0<r \leq i^{-1}$. Observe that
 $A(i,j)$ is a closed subset of $\mathbf R^n\times  \mathscr E'(\mathbf R^n,Y)_b \times Y_k$ and
\begin{equation}\label{Classifier for the limit condition of differentiability of distribution}
	E = \bigcup_{i,j = 1}^\infty  C(i,j) \cap E(i,i), \qquad
		E_0 = \bigcap_{l=1}^\infty\bigcup_{i,j=1}^\infty  C(i,j)\cap E(i,l^{-1}).
\end{equation} 
Then by \ref{lower-semicontinuous-lemma} and \ref{Continuity of embedding the space of distribution}, 
the set $E(i,\epsilon)$ is closed in $\mathbf R^n\times  \mathscr E'(\mathbf R^n,Y)_b \times Y_k$. 
We conclude from \eqref{Classifier for the limit condition of differentiability of distribution} that 
$E$ is a countable union of closed subsets of $\mathbf R^n\times  \mathscr E'(\mathbf R^n,Y)_b \times Y_k$. 
If in additional $Y$ is finite dimensional then by \ref{Bounded subset of strong topology of space of distribution} 
$A(i,j)$ is a compact subset of $\mathscr E'(\mathbf R^n,Y)_b$. This proves \ref{Borel-function-of-ptD:enum:1} and \ref{Borel-function-of-ptD:enum:2}.

To prove \ref{Borel-function-of-ptD:enum:3} and \ref{Borel-function-of-ptD:enum:4}, 
for each positive integer $i$, we construct $\zeta_i \in \mathscr D(\mathbf R^n)$ with the following two properties
\begin{gather*}
	\zeta_i(x) = 1 \quad \text{for every $x\in \mathbf B(0,2 i)$},\\
	\spt \zeta_i \subset \mathbf B(0,3i).
\end{gather*}
By \ref{Continuity of embedding the space of distribution},
 the function $\chi_i$ characterized by
 \begin{gather*}
	\chi_i : \mathscr D'(\mathbf R^n, Y)_b \to \mathscr E'(\mathbf R^n,Y)_b,\\
	\chi_i(T)(\theta) =  T(\zeta_i \theta) \quad \text{for every $\theta \in \mathscr E(\mathbf R^n,Y)$}
 \end{gather*}
is continuous. Observe that 
\[
\begin{aligned}
	D = \bigcup_{i=1}^\infty (\mathbf B(0,i)\times  \mathscr D'(\mathbf R^n,Y)_b \times Y_k ) \cap (1_{\mathbf R^n}\times\chi_i \times 1_{Y_k})^{-1} [E],\\
	 D_0 = \bigcup_{i=1}^\infty (\mathbf B(0,i)\times  \mathscr D'(\mathbf R^n,Y)_b \times Y_k ) \cap  (1_{\mathbf R^n}\times\chi_i \times 1_{Y_k})^{-1} [E_0].
\end{aligned} \]
The conclusions of \ref{Borel-function-of-ptD:enum:3} 
and \ref{Borel-function-of-ptD:enum:4} 
follow from \ref{Borel-function-of-ptD:enum:1} and \ref{Borel-function-of-ptD:enum:2}, respectively.
\end{proof}
We will use the following remark with \ref{Bounded subset of strong topology of space of distribution} in the next corollary to see that the Borel sets $\mathbf R^n\times \mathscr D'(\mathbf R^n)_s$ and $\mathbf R^n\times \mathscr D'(\mathbf R^n)_b$ agree.
\begin{remark}\label{rmk:Suslin projection}
	Suppose $F$ is a Lusin Hausdorff locally convex space and $\tau$ is the locally convex topology of $F$ (see \ref{Def:TVS:LC} for definition). For every topology $\tau'$ let $B_{\tau'}$ be the family of Borel sets with respect to $\tau'$. 
	Then $B_{\tau'} = B_\tau$ whenever $\tau'$ is a topology finer than a Hausdorff topology on $F$ and coarser than $\tau$. The conclusion follows directly from \ref{Suslin-projection}\citep[see][2.2.10]{MR41:1976}.
\end{remark}
\begin{corollary}\label{ThmA}\ref{cor:Borel reg}
	Suppose $Y$ is a Banach space, $k$ is an integer, $0< \alpha \leq 1$, $k+\alpha \geq 0$, $\nu$ is a continuous seminorm on $\mathscr D_{\mathbf B(0,1)}(\mathbf R^n,Y)$, and $T\in \mathscr D'(\mathbf R^n,Y)$.
	Let $Y_k$ be vector space of polynomial function of degree at most $k$ with values in $Y'$. Then the following two statements hold
	\begin{enumerate}[ref=\thetheorem(\theenumi)]
		\item \label{ThmA-1}The set $D'$ of all $(a,P)\in  \mathbf R^n\times Y_k$ such that $T$ is $\nu$ pointwise differentiable of order $(k,\alpha)$ with $k$ jet $P$ at $a$ is a countable union of compact subsets of $\mathbf R^n\times Y_k$.
		\item \label{ThmA-2}If $k\geq 0$, then the set $D'_0$ of all $(a,P)\in  \mathbf R^n\times Y_k$ such that $T$ is $\nu$ pointwise differentiable of order $k$ with $k$ jet $P$ at $a$ is a Borel subset of $\mathbf R^n$.
		\item \label{ThmA-3} The set of $a \in \mathbf R^n$ such that $T$ is $\nu$ pointwise differentiable of order $(k,\alpha)$ at $a$ is a countable union of compact subsets of $\mathbf R^n$.
		\item \label{ThmA-4}If $k\geq 0$ and $Y$ is separable, then $D'_0$ is a Borel function. Furthermore, $\dmn D'_0$ is a Borel set.
		\item \label{ThmA-5}If $k\geq 0$ and $Y$ is finite dimensional,
		 then for every $j\in \{0,1,\dots,k\}$ 
		 the set of all 
		 $(a,T,P) \in \mathbf R^n \times \mathscr D'(\mathbf R^n,Y) \times Y_k$ 
		 such that $T$ is $\nu$ pointwise differentiable of order $k$ at $a$
		 with $k$ get $P$
		 is a Borel set and the function mapping 
		 $(a,T) \in X$ onto $\pt \D^j T(a) = \D^j P(a)$ is a Borel function on $\mathbf R^n \times \mathscr D'(\mathbf R^n,Y)_s$ whose domain is a Borel set.
	\end{enumerate}
\end{corollary}

\begin{proof} 
	Let $D,D_0$ be the Borel set as in \ref{Borel-function-of-ptD} with $\kappa$ replaced by $-k-\alpha$, $-k$. 
	From \ref{IntroDef:Pt}, we see that $(a,T,P)$ belongs to $D$ 
	if and only if 
	$T$ is $\nu$ pointwise differentiable of order $(k,\alpha)$ at $a$, 
	and in case $k\geq0$, $(a,T,P)$ belongs to $D_0$ 
	if and only if 
	$T$ is $\nu$ pointwise differentiable of order $k$ at $a$ with $k$ jet $P$. 
	Because $\mathbf R^n$ and $Y_k$ are locally compact,
	 we conclude that $(\mathbf R^n \times \{ T\} \times Y_k) \cap D$ is a countable union of compact subsets of 
	 $\mathbf R^n \times \{ T\} \times Y_k$ and 
	 $(\mathbf R^n \times \{ T\} \times Y_k) \cap D_0$ is a Borel subset of $\mathbf R^n \times \{ T\} \times Y_k$,
	  this proves \ref{ThmA-1}, \ref{ThmA-2}, and \ref{ThmA-3}. By \ref{comparison of pointwise differentiability:1} and \cite[2.10]{2021}, the relation $D'_0$ is a function. The conclusion of \ref{ThmA-4} now follows from \ref{rmk:Suslin projection}. Finally, if $Y$ is finite dimensional, then notice that $D_0$ equals the set in the conclusion of \ref{ThmA-5}. By \ref{Borel-function-of-ptD:enum:4}, \ref{rmk:Suslin projection}, and \ref{Bounded subset of strong topology of space of distribution}, $\mathscr D'(\mathbf R^n,Y)_c$ and $\mathscr D'(\mathbf R^n,Y)_s$ has the same Borel sets. We infer that the function mapping $(a,T) \in \mathbf R^n \times \mathscr D'(\mathbf R^n,Y)$ onto the element $P = D_0(a,T)$ uniquely characterized by $(a,T,P) \in D_0$ is a Borel function with $\dmn D_0$ is a Borel set. This proves \ref{ThmA-5}.
\end{proof}

\section{Rectifiability of the family of k jets and a Lusin type theorem}\label{sec:Lusin plus Rect}

In view of the Borel regularity of $\nu$ pointwise differential, Theorem \ref{Thm-C} and Theorem \ref{introThmE} are direct consequence of \cite[4.9 and 4.25]{2021}. Theorem \ref{Thm-C} will also be used in the proof of Theorem \ref{introThmD} (see \ref{Radamacher theorem}).

\begin{theorem}\label{Rect}\ref{Thm-C}
	Suppose $Y$ is a Banach space, 
	$i,k$ are nonnegative integers,
	 $0<\alpha \leq 1$, $0 \leq k + \alpha$, 
	 $1\leq p<+\infty$,
	  $T \in \mathscr D'(\mathbf R^n,Y)$,
	   and $A$ is the set of points of $a\in \mathbf R^n$ such that
	    $T$ is $\nu_{i,p}$ pointwise differentiable of order $(k,\alpha)$ at $a$. 
		Then there exists a sequence of functions 
		$f_1,\dots,$ of class $(k,\alpha)$ mapping $\mathbf R^n$ into $Y'$ 
		and compact sets $C_1,\dots,$ such that $A = \bigcup_{i=1}^\infty C_i$ and
	 \begin{equation}
	 	   \begin{aligned}
	 		&\pt \D^m T(a) = \D^m f_j(a)\\
	 		&\text{for every positive integer $j$, $a\in C_j$, and $m\in \{ 0,\dots,k\}$}.
	 	\end{aligned}
	 \end{equation}
\end{theorem}

\begin{proof}
	Let $B$ be the set of $b\in \mathbf R^n$ such that $T$ is pointwise differentiable of order $(k,\alpha)$. We apply \cite[4.9]{2021} to obtain a sequence of functions $f_1,\dots,$ of class $(k,\alpha)$ mapping $\mathbf R^n$ into $Y'$ and compact sets $K_1,\dots,$ such that $	B = \bigcup_{i=1}^\infty K_i$ and
	\begin{equation}
		 \begin{aligned}
			&\pt \D^m T(a) = \D^m f_j(a)\\
			&\text{for every positive integer $j$, $a\in K_j$, and $m\in \{ 0,\dots,k\}$}.
		\end{aligned}
	\end{equation}
We apply \ref{ThmA} with $\nu$ replaced by $\nu_{i,p}$ to conclude that the set $A$ is a countable union of compact subset of $\mathbf R^n$. Notice that from \ref{comparison of pointwise differentiability:1} we have $A \subset B$. We rearrange the indices to obtain the conclusion.
\end{proof}

\begin{theorem}\ref{introThmE}\label{Lusin theorem}
	Suppose $Y$ is a Banach space, $Y'$ is separable, $i,k$ are nonnegative integers,
	$\nu$ is a real valued seminorm on $\mathscr D_{\mathbf B(0,1)}(\mathbf R^n,Y)$, 
	and $A$ is the set of $a\in \mathbf R^n$ at which the distribution $T$ is $\nu$ pointwise differentiable of order $k$. Then for every $0<\epsilon<+\infty$ there exist
	a function $g$ of class $k$ mapping $\mathbf R^n$ into $Y'$,
	\[ \mathscr L^n(A \sim \{ a : \pt \D^m T(a) = \D^m g(a)\quad \text{for $m=0,\dots,k$}\}) < \epsilon.\] 
\end{theorem}

\begin{proof}
   Combine \ref{comparison of pointwise differentiability:1}, \ref{Borel-function-of-ptD} and \cite[4.25]{2021}.
\end{proof}

\section{Auxiliary PDE estimates}\label{est}

In preparation for the proof of Theorem \ref{introThmD}, 
this section is dedicated to collecting some preliminary estimates 
for Sobolev functions and the poly-Laplacian.

\begin{definition}\label{Def:Lap}
	Suppose $i$ are positive integers. The differential operator $\Lap$ in $\mathbf R^n$ is written as $\D_1^2+\cdots+\D_n^2$.
\end{definition}
The following theorem is a consequence of the main result in \cite{Agmon2011},
 which gives a priori estimates for weak solutions of elliptic boundary value problems,
  and we shall use it to derive a series of results on $\Lap^i$.
\begin{theorem}\label{Dirichlet problem}\emph{\cite[Theorem 8.1]{Agmon2011}}
	\Copy{Dirichlet problem}{
	Suppose $i$ are  nonnegative integers, 
 $1<p<+\infty$, 
 $G$ is a class $2i$ bounded open subset of $\mathbf R^n$.
  Then there exists a constant $C$ depending on $n,i,p,G$ with  the following property.
   For every $1<q<+\infty$, integer $0\leq j \leq 2i$,
    and $u \in L_q(G)$ the following inequality holds
	\[ |u|_{2i-j,p;G} \leq C(|\Lap^i u|_{-j,p;G}+|u|_{p;G}).\]
	}
\end{theorem}

\begin{corollary}\label{Uniqueness of Dirichlet problem}
	Suppose $n,i,p,G$ are as in \ref{Dirichlet problem}, $i>0$, $u\in L_p(G)$, and $\Lap^i u = 0$. Then $u \in \mathscr E(\overline{G})$.
\end{corollary}
\begin{proof}
	We first consider the case $n<p$.
	 Inductively apply  \ref{Dirichlet problem} to infer
	  $u\in \bigcap_{i=0}^\infty W_{i,p}(G)$. The assertion follows from the Rellich compactness theorem \cite[Theorem 7.26]{MR1814364}.
	
	Now, we show that the second case $p<n$ can be reduced to the first case. 
	In view of \ref{Dirichlet problem}, $u \in W_{2i,p}(G)$. 
	By \cite[Theorem 7.10 and 7.25]{MR1814364} we know that $|u|_{\frac{np}{n-p};G}<+\infty$. 
	We repeat the argument with $p$ replace by $\frac{np}{n-p}$, iterating if necessary until we obtain an exponent $p$ strictly greater than $n$. After finitely many steps, it is reduced to the first case. 

\end{proof}
\begin{corollary}\label{Prior estimate of Dirichlet problem on balls}
	Suppose $i,k$ are nonnegative integers, 
	$1<p<+\infty$. 
	Then there exists a constant $C$ depending on $n,i,k,p$ with the following property. 
	For every $0<r<+\infty$, $1<q<+\infty$, and $ u \in L_q(\mathbf U(0,r)) $, 
	the following inequality holds, 
	\[ |\D^{2i+k}u|_{p;0,r} \leq C (|\D^k\Lap^i u|_{p;0,r}+ r^{-2i-k}|u|_{p;0,r}).\]
\end{corollary}

\begin{proof}
	Assume $a = 0$, $r =1$.
	We apply \ref{Dirichlet problem} with $G$ replaced by $\mathbf U(0,1)$ to obtain the constant the case $k=0$.
	Suppose the results holds for $k-1$ and $k> 0$ with the 
	constants $C_{\ref{Prior estimate of Dirichlet problem on balls}}(n,i,p,k-1)$ and $|\Lap^i u|_{k,p;0,1}<+\infty$. 
	From \ref{append:Interpolation inequality for Balls}, we have $u\in W_{2i+k-1,p}(\mathbf U(0,1))$.
	 For every $0<\lambda <1$, $v\in \mathbf R^n$, and $0<|v| <1- \lambda$ we define $u_v \in W_{2i+k-1,p}(\mathbf U(0,\lambda))$ by
	\begin{equation}
		u_v(x) = \frac{u(x+v) -  u(x)}{|v|} \quad \text{for $\mathscr L^n$ almost every $x\in \mathbf U(0,\lambda)$}.
	\end{equation}
	Using the equation $\Lap^i u_v = (\Lap^i u)_v$, we estimate by inductive hypothesis, 
	\[ \begin{aligned}|\D^{2i+k-1}u_v|_{p;a,\lambda} \leq
		  C_{\ref{Prior estimate of Dirichlet problem on balls}} (|\D^{k-1}(\Lap^iu)_v|_{p;a,\lambda} 
		+ \lambda^{-2i-k+1} |u_v|_{p;a,\lambda}).\end{aligned}\]
	Since $1<p<+\infty$, we may let $v \to 0$ to conclude that $u$ is $2i+k$ times weakly differentiable on $\mathbf U(0,\lambda)$ and
	\begin{equation}
		\label{Prior estimate of Dirichlet problem on balls:estimate:1}
		\begin{aligned}
		|\D^{2i+k}u|_{p;a,\lambda}\leq & n C_{\ref{Prior estimate of Dirichlet problem on balls}}(|\D^{k}\Lap^i u|_{p;a,\lambda} 
		+ \lambda^{-2i-k+1} |\D u|_{p;a,\lambda}).
	\end{aligned}
	\end{equation}
Combine \ref{Prior estimate of Dirichlet problem on balls:estimate:1}
 and \ref{append:Interpolation inequality for Balls}, and let $\lambda \uparrow 0$,
  to obtain the conclusion.
\end{proof}
\begin{corollary}\label{Prior estimate of Dirichlet problem on balls for zero boundary value function}
	Suppose $n,i$ are nonnegative integers, $0<i$, $1<p<+\infty$, and $G$ is a bounded open subset of class $2i$ in $\mathbf R^n$. Then there exists a constant $C$ depending on $n,i,p,G$ with the following property. For every $u\in W_{2i-j,p}^\diamond(G)$ and $j\in \{ 0,1,\dots,2i-1\}$, 
	\[ |u|_{2i-j,p;G} \leq C |\Lap^iu|_{-j,p;G}.\]
\end{corollary}

\begin{proof}
	
	We apply \ref{Dirichlet problem} to obtain $C_{\ref{Dirichlet problem}}(n,i,p,G)$.
	If the inequality does not hold, then there will exists $j\in \{ 0,1,\dots,2i-1\}$ and a sequence $u_1,\dots,\in  W_{i,p}^\diamond(G)$ such that
	\[  |u_k|_{2i-j,p;G} >k |\Lap^i u_k|_{-j,p;G}, \quad |u_k|_{p;G} =1 \quad \text{for $k\in \{1,2,\dots,\}$}.\]
	By \ref{Dirichlet problem}, for those positive integer $k$ strictly greater than $2C_{\ref{Dirichlet problem}}$  we have 
	\begin{align*}
		|u_k|_{2i-j,p;G} &\leq C_{\ref{Dirichlet problem}}(|\Lap^i u_k|_{-j,p;G} + |u_k|_{p;G})\\
						 &\leq C_{\ref{Dirichlet problem}}(k^{-1}|u_k|_{2i-j,p;G} + |u_k|_{p;G})\\
						 &\leq C_{\ref{Dirichlet problem}}((2C_{\ref{Dirichlet problem}})^{-1}|u_k|_{2i-j,p;G} + |u_k|_{p;G})
	\end{align*} 
	and consequently
	the sequence $u_1,\dots$ forms a bounded subset of $W_{2i-j,p}(G)$.
	By the Rellich compactness theorem \cite[Theorem 7.26]{MR1814364}, we may assume $u_1,\dots,$ converges to $u$ in $W_{2i-j-1,p}^\diamond(G)$.
	By \ref{Uniqueness of Dirichlet problem}, we also have $u \in \mathscr E(\overline{G})$. Therefore 
	$u \in  W_{2i,2}^\diamond(G),  |u|_{p;G}=1$,
	but $\lim\limits_{k\to \infty}|\Lap^i u_k|_{-j,p;G } = 0$ implies $\Lap^i u=0 $ and  $u=0$, which contradicts to $|u|_{p;G}=1$.
\end{proof}

\begin{corollary} \label{Dirichlet problem on unit ball}
	Suppose $n,i$ are nonnegative integers, $1<p<+\infty$, $0<R<+\infty$, and $a \in\mathbf R^n$. For every $f \in W_{i,q}^{\diamond}(\mathbf U(a,R))'$ there exists an $\mathscr L^n$ almost  unique $u \in W_{i,p}^\diamond(\mathbf U(a,R))$ such that $\Lap^i u = f$ and
	\[ |\D^i u|_{p;a,R} \leq C |\Lap^i u|_{-i,p;a,R},\]
	where $p^{-1}+q^{-1} = 1$ and $0<C<+\infty$ depends on $n,i,p$.
\end{corollary}

\begin{proof}
	Assume $0<i$, $a=0$, and $R=1$. We apply \ref{Prior estimate of Dirichlet problem on balls for zero boundary value function} to obtain $C_{\ref{Prior estimate of Dirichlet problem on balls for zero boundary value function}}(n,i,p,\mathbf U(0,1))$. In view of \ref{Prior estimate of Dirichlet problem on balls for zero boundary value function} and the fact that the spaces $W_{i,p}^
	\diamond(\mathbf U(0,1))$ are reflexive, it is sufficient to prove existence of $u$ for the case $p<2$. 
	 Next, we approximate $f$ by a sequence $f_1,\dots,\in W_{i,2}^\diamond(\mathbf U(0,1))'$ 
	 in $W_{i,q}^\diamond(\mathbf U(0,1))'$,
	 we will show that such approximating sequence must exists,
	 by Hahn-Banach theorem it is sufficient to show that an element $\alpha \in W^\diamond_{i,q}(\mathbf U(0,1))''$
	 vanishes if and only if $\alpha$ vanishes on $W_{i,2}^\diamond(\mathbf U(0,1))'$,
	 this follows from $W_{i,q}(\mathbf U(0,1))$ is reflexive.
	  In view of \ref{Poincare inequality}, the topology on $W_{i,2}^
	 \diamond(\mathbf U(0,1))$ is induced by the following inner product, 
	 \[ \textstyle u \bullet v = \int_{\mathbf U(0,1)} \D^i u \bullet \D^i v \ d \mathscr L^n \ \text{for every $u,v\in W_{i,2}^\diamond(\mathbf U(0,1))$}.\]
	 It follows that there exists a sequence $u_1,\dots, \in W_{i,2}^\diamond(\mathbf U(0,1))$ such that
	\[ \Lap^i u_k = f_k \quad \text{for $k \in \{1,2,\dots\}$}.\]
	
	Notice that $W_{i,2}^\diamond(\mathbf U(0,1)) \subset W_{i,p}^\diamond(\mathbf U(0,1))$.
	We apply \ref{Prior estimate of Dirichlet problem on balls for zero boundary value function} to estimate
		\[ |\D^i(u_k - u_{k'})|_{p;0,1} \leq C_{\ref{Prior estimate of Dirichlet problem on balls for zero boundary value function}} |f_k - f_{k'}|_{-i,p;0,1} \quad \text{for $k,k' \in \{1,2,\dots,\}$}.\]
	Consequently, 
	the convergent sequence $f_1,f_2,\dots,$ implies that there exists a $u\in W_{i,p}^\diamond(G)$, 
	such that $\lim\limits_{i\to\infty}|u_k - u|_{i,p;G} = 0$. This $u$ satisfies the conclusion.
\end{proof}

\section{A Rademacher  type theorem}\label{Rade_Pt}
First we prove the theorem assuming $T$ has $0$ $k$ jet,
on $A$ (see \ref{Radamacher theorem for distribution with k jet equals 0}).
Roughly speaking, this is done by decomposing $A$ into compact subsets $A_s$
on which we may obtain a representation of $T$ as $\Lap^i u$ for $u\in W_{i,p}^\diamond(\mathbf U(0,R_0))$.
These subset are themselves decomposed into subsets $B^s_l$ on which $u$ is 
$\nu_{i,p}$ pointwise differentiable with uniform estimates (see \ref{Uniform decay on the difference}).
For each $B_{s,l}$, the uniform estimates allow us to find polyharmonic functions
 $v_{a,r} \in W_{i,p}^\diamond(\mathbf U(0,1))$ which approximat $u$ with good \textit{priori}
  estimates. In \ref{Pointwise comparison decay to global comparison decay}, 
  it is shown that theise approximating functions may be patched together to obtain 
  into a function $v$ which is of class $W_{i,p}^{loc}$ in a neighborhood of $B^s_l$,
  and which ingherites the a \textit{priori} estimates of the $v_{a,r}$.
  The function $v$ allow as to invoke Calder\'on's theorem \cite[Theorem 10]{MR0136849}
   and then \ref{Rešhetnyak’s} to finish the proof of the theorem.
   The full theorem \ref{introThmD} is then proven in \ref{Radamacher theorem} 
   by invoking \ref{Rect} to show that we can reduce to 
   \ref{Radamacher theorem for distribution with k jet equals 0}
   by substracting a function of class $k$ from $T$.
\begin{lemma}\label{Pointwise comparison decay to global comparison decay}
	Suppose $1< p<+\infty$, $0 \leq \lambda <+\infty$,
	 and $i$ is a nonnegative integer. 
	  Then there exists a constant $\Delta$ with the following property. 
	  If $A$ is a closed subset of $\mathbf R^n$, 
	  $U$ is an open subset,
	   $0<M,\gamma<+\infty$, 
	   $ i \leq \lambda <+\infty$, 
$ u \in W_{i,p}(U)$,
$0<\kappa <\inf\{ |x-a| : a\in A,x\in \mathbf R^n \sim U\}$, 
$k$ is the largest nonnegative integer less than or equal to $\lambda -i$,
 and for every $a\in A$ and $0<r\leq \kappa$,
  there exist $v_{a,r} \in W_{i,p}(\mathbf U(a,r))$ and 
  polynomial function $P_a$ with value in $\mathbf R$ of degree at most $2i+k-1$ such that
\begin{equation}\label{Pointwise comparison decay to global comparison decay:pointwise decay condition}
	\begin{aligned}
		&u-v_{a,r} \in W_{i,p}^\diamond(\mathbf U(a,r)),\quad
		r^{-n/p-i} |u- P_a|_{p;a,r} \leq \gamma ,\\
		&r^{-n/p} |\D^i(u-v_{a,r})|_{p;a,r} \leq M r^\lambda ,\quad
		\Lap^i v_{a,r} =0.
	\end{aligned}
\end{equation}
Then there exists 
$v \in W_{2i+k,p}^{loc}(\mathbf R^n \cap \{ x : \dist(x,A)<\kappa/36\})$ 
such that for every $a\in A$, $0\leq j \leq i$, and $0<r\leq \kappa/72$
\begin{equation}
	|\D^j(u-v)|_{p;a,r} \leq \Delta M r^{n/p+\lambda+i-j}.
\end{equation}
\end{lemma}
\def\FuncAccent#1#2{#1_{#2}}
\begin{proof}
	Choice of constants
	\[C_{\ref{append:Interpolation inequality for Balls}}(n,k,p,1), \ C_{\ref{append:Poincare inequality for zero boundary value function on balls}}(n,2i)
	, \ C_{\ref{Prior estimate of Dirichlet problem on balls}}(n,i,k,p), \ C_{\ref{Dirichlet problem on unit ball}}(n,i,p). \] 
	In the proof, constant depends on $(n,i,p,\lambda)$ is denoted by $\Delta_1,\Delta_2,\dots$

	For every $0<\delta \leq \kappa/18$, we
	apply \citep[3.13]{MR41:1976} with $\Phi = \{ \mathbf R^n \sim A\} \cup \{ \mathbf U(a,\delta) : a\in A\}$, to obtain
	a function $h$ defined by 
	\begin{equation}\label{Def of h}
		h(x) = \frac{1}{20} \sup \{ \inf \{ 1 , \dist(x,\mathbf R^n \sim T)\} : T \in \Phi\} \quad \text{for every $x\in \mathbf R^n$},
	\end{equation}
	a countable set $S \subset \mathbf R^n$, functions $\zeta_s$ corresponding to $s\in S$ such that
	\begin{equation*}
		\begin{aligned}
			&\{ \mathbf B(s,h(s)) : s\in S\} \quad \text{is disjointed},\quad
			\mathbf R^n = \bigcup \{ \mathbf B(s,5h(s)) : s\in S\} ,\\
			&\zeta_s \in \mathscr D_{\mathbf B(s,10h(s))}(\mathbf R^n) \quad \text{for every $s\in S$},
		\end{aligned}
	\end{equation*}
	 constants $0\leq V_0,V_1,\dots,<+\infty$ depends on $n$, and function $\xi$ mapping $\mathbf R^n \sim A$ to $A$ such that
	\begin{equation}\label{Pointwise comparison decay to global comparison decay:partition of unity}
		\begin{aligned}
		&|\D^k \zeta_s(x)| \leq V_i h(x)^{-k} \quad  \text{for every $x\in \mathbf R^n$ and $s\in S$},\\
		&|\xi(s) -s| = \dist(s,A) \quad \text{for every $s\in \mathbf R^n \sim A$}.
		\end{aligned}
 	\end{equation}
 For every $x\in \mathbf R^n$, let $ S_x = S \cap \{ s : \mathbf B(s,10h(s)) \cap \mathbf B(x,10h(x))\}$.
 Note $\card S_x \leq (129)^n$, $\sum_{s\in S_x} \zeta_s(x) = 1$, and for $s\in S_x$,
 $1/3\leq \frac{h(x)}{h(s)} \leq 3$. 
 	After verifying (see \eqref{Def of h} and \cite[first paragraph of proof of 2.8]{e5d4609f3eb74629ad6704c7dc209467}),
 	\begin{equation}\label{Pointwise comparison decay to global comparison decay:verified:1}
 		\begin{aligned}
 			&20 h(x) \leq \sup \{ \dist(x,A) ,\delta\} \quad \text{for every $x\in \mathbf R^n$},\\
 			&\mathbf B(x,20 h(x)) \subset \mathbf B(\xi(s),120 h(s)) \\ 
 			&\text{for every $x\in \mathbf R^n$, $\dist(x,A) \leq \frac{\kappa}{18}$ and $s\in S_x$},
 		\end{aligned}
 	\end{equation}
 	set $R = \bigcup \{ S_x : x\in \mathbf R^n, \dist(x,A) \leq \kappa/18\}$ and define for every $s\in R$, functions  $v_s,P_s$ and a function $\FuncAccent{w}{\delta}$ depends on $\delta$ by the following equations.
 	\begin{equation}
 		\begin{aligned}	
 			&v_s(x) =\begin{cases}
 				v_{\xi(s),120h(s)} &: x\in \mathbf U(\xi(s),120 h(s)),\\
 				0 &: x \in \mathbf R^n \sim \mathbf U(\xi(s),120 h(s)) ,
 			\end{cases} \\
 			&w_\delta(x)= \sum_{s\in S} \zeta_s(x)v_s(x) \quad \text{for every $x\in \mathbf R^n$}, \qquad P_s = P_{\xi(s)}.
 		\end{aligned}
 	\end{equation}

For every $x\in \mathbf R^n$ with $\dist(x,A) \leq \kappa/18$, $j\in \{ 0,1,\dots,i\}$, and $s,t\in S_x$,  we estimate by 
 \ref{append:Poincare inequality for zero boundary value function on balls} and \eqref{Pointwise comparison decay to global comparison decay:pointwise decay condition} that
\begin{equation}\label{Estimate of comparison to the target object}
	\begin{aligned}
		|\D^j(u-v_t)|_{p;x,20h(x)} &\leq |\D^j(u-v_t)|_{p;s,120h(s)}\\
								   &\leq C_{\ref{append:Poincare inequality for zero boundary value function on balls}} [120h(s)]^{i-j}|\D^i(u-v_t)|_{p;s,120h(s)}\\
								   &\leq C_{\ref{append:Poincare inequality for zero boundary value function on balls}}[120 h(s)]^{i-j+n/p+\lambda}M\\
								   &\leq\Delta_1 h(x)^{i-j+n/p+\lambda}M,
\end{aligned}
\end{equation}
where $\Delta_1 = C_{\ref{append:Poincare inequality for zero boundary value function on balls}} 360^{i+n/p + \lambda}$.

For every $x\in \mathbf R^n$ with $\dist(x,A) \leq \kappa/18$, $j\in \{ 0,1,\dots,2i+k\}$, and $s,t\in S_x$
we estimate by \eqref{Pointwise comparison decay to global comparison decay:pointwise decay condition}, \ref{Dirichlet problem on unit ball}, and \ref{Estimate of comparison to the target object} that
\begin{equation}\label{comparison to the boundary value at different ball}
	\begin{aligned}
		|\D^{j} (v_s-v_t)|_{p;x,20h(x)} &\leq C_{\ref{Dirichlet problem on unit ball}} [20 h(x)]^{-j}|v_s - v_t|_{p;x,20 h(x)}\\
										&\leq C_{\ref{Dirichlet problem on unit ball}} [20 h(x)]^{-j} (|v_s - u|_{p;x,20 h(x)} + |u-v_t|_{p;x,20 h(x)})\\
										& \leq C_{\ref{Dirichlet problem on unit ball}}[20 h(x)]^{-j} 2 \Delta_1 h(x)^{i-j+n/p+\lambda}M\\
		&\leq \Delta_2 h(x)^{i-j+n/p+\lambda}M,
	\end{aligned}
\end{equation}
where $\Delta_2 = C_{\ref{Dirichlet problem on unit ball}}  2 \Delta_1 $.

Observe that for every $x\in \mathbf R^n$ with $\dist(x,A) \leq \kappa/18$,
 	\begin{equation}\label{Partition of unity representation}
 		\begin{aligned}
 				&u(y) - w_\delta(y) = \sum_{s\in S_x} \zeta_s(y)[u(y) - v_s(y)] ,\\
 			&  w_\delta(y)-v_s(y) = \sum_{t\in S_x} \zeta_s(y)[v_t(y)-v_s(y)]\\
 			&\text{for every $y\in \mathbf B(x,10 h(x))$}.
 		\end{aligned} 
 	\end{equation}
	For every $x\in \mathbf R^n$ with $\dist(x,A) \leq \kappa/18$ and $j\in \{ 0,1,\dots,i\}$,
 	using the first equation of \eqref{Partition of unity representation}, \ref{defpar:Symmetric algebra}, Leibniz's formula, \eqref{Pointwise comparison decay to global comparison decay:partition of unity}, \eqref{Estimate of comparison to the target object}, and \eqref{Pointwise comparison decay to global comparison decay:pointwise decay condition}, we estimate
 	\begin{equation}\label{Local comparison at point near A:end of claim}
 		\begin{aligned}
 			|\D^j(u-w_\delta)|_{p;x,10h(x)} &\leq \sum_{l=0}^j \sum_{s\in S_x} |\D^l \zeta_s \odot [ \D^{j-l} (u-v_s)]|_{p;x,10h(x)} ,\\
 									 &\leq \sum_{l=0}^j \sum_{s\in S_x} V_l 9^{-l}h(x)^{-l} \Delta_1 h(x)^{i-(j-l)+n/p+\lambda}M , \\
 									 &\leq \Delta_3 M h(x)^{n/p+\lambda+i-j},
 		\end{aligned} 
 	\end{equation}
 where $\Delta_3 = \sum_{l=0}^{i} (129)^n V_l 9^{-l} \Delta_1 $. 

 For every $x\in \mathbf R^n$ with $\dist(x,A) \leq \kappa/18$ and $j\in \{ 0,1,\dots,2i+k\}$,
 using the second equation of \eqref{Partition of unity representation}, \ref{defpar:Symmetric algebra}, Leibniz's formula, \eqref{Pointwise comparison decay to global comparison decay:partition of unity}, and \eqref{comparison to the boundary value at different ball} we estimate 
\begin{equation}
\begin{aligned}\label{2i order decay:estimate:2}
	|\D^{j}(w_\delta-v_s)|_{p,x,10h(x)} &\leq \sum_{t\in S_x} \sum_{l=0}^{j} |\D^l \zeta_t \odot [ \D^{j-l} (v_t-v_s)]|_{p;x,10h(x)}\\
								  &\leq \sum_{t\in S_x} \sum_{l=0}^{j} V_l 9^{-l}h(x)^{-l} \Delta_2 h(x)^{i-(j-l)+n/p+\lambda}M\\
								  &\leq  \Delta_4 h(x)^{i-j+n/p+\lambda}M,
\end{aligned}
	\end{equation}
where $\Delta_4 = (129)^n\sum_{l=0}^{2i+k} V_l 9^{-l}\Delta_2$.

Next, we claim that for each $a\in A$, $\delta \leq r\leq \kappa/18$, and $j \in \{ 0,1\dots,i\}$  
\begin{equation}\label{Local comparison at point of A:claim}
	\begin{aligned}
		&|\D^j(u-w_\delta)|_{p;a,r}\leq M\Delta_5 r^{n/p+\lambda +i-j},\\ 
		&|\D^{2i+k}w_\delta|_{p;a,r} \leq \Delta_7 (M+\gamma)r^{n/p+\lambda-i-k}.
	\end{aligned} 
\end{equation}
 To see the claim, notice that if $x\in \mathbf B(a,r)$, then
$\mathbf B(x,20 h(x)) \subset \mathbf B(a,2r)$. Thus, we may apply the Vitali covering theorem \cite[2.8.4]{MR41:1976} with $\delta,F,\tau$ by $\diam,\{ \mathbf B(x,2h(x)) : x\in \mathbf B(a,r)\},2$ to obtain a countable set $T \subset \mathbf B(a,r)$ such that 
\[  \{ \mathbf B(t,2h(t)) : t\in T \}\quad \text{is disjointed} \quad \text{and}\quad \mathbf B(a,r) \subset \bigcup \{ \mathbf B(t,10h(t)) : t \in T\}. \]
By \eqref{Local comparison at point near A:end of claim}, we have
\begin{equation}\label{Local comparison at point of A:end of claim1}
	\begin{aligned}
	|\D^j(u-w_\delta)|^p_{p;a,r} &\leq \sum_{t\in T}|\D^j(u-w_\delta)|^p_{p;t,10h(t)} ,\\
							&\leq \sum_{t\in T} [\Delta_3 M h(t)^{n/p+\lambda+i-j}]^p\\
							&=  [\Delta_3 M]^{p} \boldsymbol{\alpha}(n)^{-1} 2^{-n}\sum_{t\in T}
							 \mathscr L^n(\mathbf B(t,2h(t))) h(t)^{(\lambda+i-j) p}\\
							&\leq [\Delta_3 M]^p r^{(\lambda +i-j)^p} \boldsymbol{\alpha}(n)^{-1}2^{-n}
							\cdot \mathscr L^n(\mathbf B(a,2r))\\
							& \leq [\Delta_5 M]^p r^{n+(\lambda+i-j) p}
	\end{aligned}
\end{equation}
where $\Delta_5 = \Delta_3$.
By \eqref{2i order decay:estimate:2} and note $\lambda - i -k\geq 0$, we estimate for every $\alpha \in 
\Xi(n,k)$,
\begin{equation}\label{Local comparison at point of A:end of claim2}
	\begin{aligned}
		|\D^\alpha\Lap^i w_\delta|^p_{p;a,r} &\leq \sum_{t\in T} |\D^\alpha\Lap^i (w_\delta-v_t)|^p_{p;t,10h(t)}\\
							 &\leq \sum_{t\in T} ni |\D^{2i+k} (w_\delta-v_t)|^p_{p;t,10h(t)}\\
							 &\leq \sum_{t\in T} ni \Delta_4^p h(t)^{(-i+n/p+\lambda-k)p}M^p\\
							 &\leq \Delta_6^p M^p  r^{n+(-i+\lambda-k)p},\\
	\end{aligned}
\end{equation}
where $\Delta_6= (ni)^{1/p}\Delta_4$.
Using \ref{Prior estimate of Dirichlet problem on balls}, \eqref{Pointwise comparison decay to global comparison decay:pointwise decay condition}, \eqref{Local comparison at point of A:end of claim1}, and \eqref{Local comparison at point of A:end of claim2}, we estimate for every $\alpha \in 
\Xi(n,k)$,
\begin{equation*}\label{2i th order derivative of v}
	\begin{aligned}
		|\D^{2i} \D^\alpha w_\delta|_{p;a,r} &\leq C_{\ref{Prior estimate of Dirichlet problem on balls}}  (|\Lap^i \D^\alpha w_\delta|_{p;a,r} + r^{-2i-k}|w_\delta-P_a|_{p;a,r}) \\
		     				  &\leq  C_{\ref{Prior estimate of Dirichlet problem on balls}} [M  \Delta_6 r^{-i+n/p+\lambda-k} +
							   r^{-2i-k}(|w_\delta-u |_{p;a,r} +|u-P_a|_{p;a,r})] \\
		     				  &\leq \Delta_7 (M  + \gamma )r^{n/p+ \lambda - i-k} ,
	\end{aligned}
\end{equation*}
where $\Delta_7 = C_{\ref{Prior estimate of Dirichlet problem on balls}} \cdot 3\cdot (\Delta_6 + \Delta_3 + 1)$.

By \ref{Poincare inequality} and \eqref{Local comparison at point of A:claim},
 the set of functions consist of the restriction of $w_\delta$ to $\mathbf R^n \cap \{ x : \dist(x,A) < \kappa /36\}$ 
 corresponding to $0<\delta \leq \kappa/18$ is bounded in $W_{2i+k,p}^{loc}(\mathbf R^n \cap \{ x : \dist(x,A) < \kappa /36\})$. 
 By Rellich compactness theorem \cite[Theorem 7.26]{MR1814364}, it contains a sequence $w_{\delta_1},w_{\delta_2},\dots,$ converges weakly to a function $w$
  in $W_{2i+k,p}^{loc}(\mathbf R^n \cap \{ x : \dist(x,A) < \kappa /36\})$ which satisfies the following property. For every $a \in A$,
$
	\lim\limits_{l \to\infty}|w_{\delta_l} - w|_{2i+k-1,p;a,\kappa/72} = 0$.
It follows that,
for every $a\in A$, $0\leq j \leq i$ and $0<\delta \leq r\leq \kappa/72$, we have
\begin{equation}\label{Desired approximation}
	\begin{aligned}
		|\D^j(u-w)|_{p;a,r} &\leq |\D^j(u-w_{\delta_l})|_{p;a,r}  + |\D^j(w-w_{\delta_l})|_{p;a,r} \\
		                      &\leq  \Delta_3 M r^{n/p+\lambda+i-j} + |\D^j(w-w_{\delta_l})|_{p;a,r}. 
	\end{aligned}
\end{equation}
Letting $l \to \infty$ to obtain the conclusion.
\end{proof}

\begin{theorem}\label{Radamacher theorem for distribution with k jet equals 0}
	Suppose $n$ is a positive integer, $i,k$ are nonnegative integers, $1<p<\infty$,
	$T \in \mathscr D'(\mathbf R^n)$ and $A$ is the set of point $a\in\mathbf R^n$ such that
	 the distribution $T$ is $\nu_{i,p}$ pointwise differentiable of order $(k-1,1)$ and  $\pt \D^m T(a) = 0$ 
	 for $m\in \{0,1,\dots,k-1\}$. 
	 Then for $\mathscr L^n$ almost every point $a \in A$ the distribution $T$ is $\nu_{i,p}$ 
	pointwise differentiable of order $k$ at $a$.
\end{theorem}

\begin{proof}
	Apply \ref{Dirichlet problem on unit ball}
	to obtain constant $C_{\ref{Dirichlet problem on unit ball}}$  as in \ref{Dirichlet problem on unit ball} depending on $n,i,p,k$, and let $q = p/(p-1)$.
For every positive integer $s$ let $A_s$ be the set of  $a \in \mathbf R^n$ such that
	\begin{equation*}
		\begin{aligned}
			|T_x(\theta(r^{-1}(x-a)))| \leq s r^{n+k} \nu_{i,p}(\theta) \\
			 \text{for every $\theta\in \mathscr D_{\mathbf B(0,1)}(\mathbf R^n)$ and $0<r\leq s^{-1}$},
		\end{aligned}
	\end{equation*}
     the above condition is equivalent to
     \begin{equation*}\label{Pointwise decay bound}
     			r^{-n/p-i-k}|T|_{p;a,r} \leq s \quad
     		\text{for every $\theta\in \mathscr D_{\mathbf B(0,1)}(\mathbf R^n)$ and $0<r\leq s^{-1}$}. 	
     \end{equation*}
By \ref{lower-semicontinuous-lemma}, $A_s$ is closed.
	We will show that the distribution $T$ is $\nu_{i,p}$ pointwise differentiable of order $k$ at $a$ for $\mathscr L^n$ almost every point $a$ of $A_s$.
	
	Let $K$ be a compact subset of $\mathbf R^n$, and write $K_s = A_s \cap K$.
	Cutting $T$ outside a neighborhood of $K_j$, we may assume there exists  positive numbers $R_0,\kappa$ such that
	\begin{gather}
		0<\kappa <\inf\{ |x-a| : a \in K_s , x\in \mathbf R^n \sim \mathbf U(0,R_0)\}, \quad |T|_{-i,p;0,R_0}<+\infty.
	\end{gather}
	 We apply \ref{Dirichlet problem on unit ball} with $R$, $a$, and $f$ replaced by $R_0$, $0$ and the linear function obtain from extending $T|\mathscr D(\mathbf U(0,R_0))$ to $W_{i,q}^\diamond(\mathbf U(0,R_0))$ to obtain $u \in W_{i,p}^\diamond(\mathbf U(0,R_0))$ with 
	 \begin{equation}
	 	(\Lap^i u )(\theta) = T(\theta) \quad \text{for every $\theta \in \mathscr D(\mathbf U(0,R_0))$}.
	 \end{equation}
 For every positive integer $l$, let $B_l$ be the set of $a \in \mathbf U(0,R_0)$ such that there exist a polynomial function  $P_a$ mapping $\mathbf R^n$ into $\mathbf R$ of degree at most $i$ such that 
\begin{equation}\label{Uniform decay on the difference}
	r^{-n/p-i} |u-P_a|_{p;a,r} \leq l \quad \text{for every $0<r\leq l^{-1}$}.
	\end{equation}
	From \ref{Rešhetnyak’s}, we know that $\mathscr L^n(\mathbf U(0,R_0) \sim \bigcup_{l=1}^\infty B_l) = 0$.
For each $a\in K_s $ and $0<r\leq s^{-1}$, we apply \ref{Dirichlet problem on unit ball} to obtain $v_{a,r} \in W_{i,p}(\mathbf U(a,r))$ such that
\begin{equation}\label{Solving the Dirichlet problem}
	\Lap^{i}(u-v_{a,r}) = T ,\quad u-v_{a,r}\in W_{i,p}^\diamond(\mathbf U(a,r)).
\end{equation}

It follows that for the point $a$ of the closed set $ B_l \cap K_s$ and $0<r\leq \inf\{ s^{-1},l^{-1},\kappa\}$, there exists $v_{a,r} \in W_{i,p}(\mathbf U(a,r))$ with following property,
\begin{equation}\label{The combination of rearragement}
	\begin{aligned}
		\Lap^{i}(u-v_{a,r}) = T ,\quad u-v_{a,r}\in W_{i,p}^\diamond(\mathbf U(a,r)),\\
		r^{-n/p-i-k}|T|_{p;a,r} \leq s, \quad r^{-n/p-i} |u-P_a|_{p;a,r} \leq l,
	\end{aligned}
\end{equation}
by \ref{Dirichlet problem on unit ball}, we have $|\D^i(u-v_{a,r})|_{p;a,r} \leq C_{\ref{Dirichlet problem on unit ball}} |T|_{-i,p;a,r}$

We apply \ref{Pointwise comparison decay to global comparison decay} with \def\halfquad{\hskip 2 pt}
\[A,\halfquad   M,\halfquad \gamma,\halfquad  \kappa,\halfquad U, \halfquad \text{and} \halfquad\lambda\]
 replaced by 
 \[B_l \cap A_s, \halfquad  C_{\ref{Dirichlet problem on unit ball}}\cdot s,\halfquad l,\halfquad \inf\{s^{-1},l^{-1},\kappa\}, \halfquad\mathbf U(0,R_0), \halfquad \text{and} \halfquad i+k\]
  to obtain a function $v \in W_{2i,p}^{loc}(\mathbf R^n \cap \{ x : \dist(x,A)<\inf\{s^{-1},l^{-1},\kappa\}/36\})$ and a constant $\Delta$ such that for every $a\in B_l \cap A_s$, $j\in \{0,1,\dots,i\}$, and $0<r\leq \inf\{s^{-1},l^{-1},\kappa\}/72$, 
\begin{equation*}
	|\D^j(u-v)|_{p;a,r} \leq \Delta M r^{n/p+2i -j+k}.
\end{equation*}
By \cite[Theorem 10]{MR0136849}, for $\mathscr L^n$ almost every $a\in B_l \cap A_s$, 
\begin{equation*}\label{u,v has the same 2i+k-j order value}
	\lim\limits_{r \to 0^+} r^{j-2i-k-n/p}|\D^j(u-v)|_{p;a,r} = 0.
\end{equation*}
Also by \ref{Rešhetnyak’s}, for $\mathscr L^n$ almost every $a\in B_l \cap A_s$, there exists a polynomial function $Q_a$ of degree at most $2i+k$ such that,
\begin{equation*}\label{2i + k jet as 2i+k-j order value}
	\lim\limits_{r \to 0^+} r^{j-2i-k-n/p}|\D^j(Q_a-v)|_{p;a,r} = 0 \quad \text{for $j\in \{0,1,\dots,2i+k\}$}.
\end{equation*}
Finally, since
\begin{gather*} 
	|\Lap^i (u-Q_a)|_{-i,p;a,r} \leq |\D^i(u-Q_a)|_{p;a,r}  \\
	\leq |\D^i(u-v)|_{p;a,r} + |\D^i(Q_a - v)|_{p;a,r} , 
\end{gather*}
	we conclude
	\begin{equation}
		\lim\limits_{r\to 0^+}r^{-i-k-n/p} |T - \Lap^i Q_a|_{-i,p;a,r} = 0.
	\end{equation}
\end{proof}

\begin{corollary}\label{Radamacher theorem}\ref{introThmD}
	Suppose $n$ is a positive integer, $i,k$ are nonnegative integers, $1<p<+\infty$, $T \in \mathscr D'(\mathbf R^n)$, and $A$ is the set of all $a\in \mathbf R^n$ such that $T$ is $\nu_{i,p}$ pointwise differentiable of order $(k-1,1)$ at $a$. Then $T$ is $\nu_{i,p}$ pointwise differentiable of order $k$ at $\mathscr L^n$ almost all $a\in A$.
\end{corollary}

\begin{proof}
	In view of \ref{Radamacher theorem for distribution with k jet equals 0},
	 we may assume $k\geq 1$. 
	 We will reduce the problem to \ref{Radamacher theorem for distribution with k jet equals 0} 
	 by the following argument. We apply \ref{Rect} with $(,\alpha)$ replaced by $(k-1,1)$ to obtain a sequence of functions $f_1,\dots,$ of class $(k-1,1)$ mapping $\mathbf R^n$ into $\mathbf R$ and compact set $C_1,\dots,$ such that $A = \bigcup_{i=1}^\infty C_i$ and
	\begin{equation}
		  \begin{aligned}
			&\pt \D^m T(a) = \D^m f_j(a)\\
			&\text{for every positive integer $j$, $a\in C_j$, and $m\in \{ 0,\dots,k\}$}.
		\end{aligned}
	\end{equation}
	By \ref{Rešhetnyak’s}, for each positive integer $i$ the distribution $S_i \in \mathscr D'(\mathbf R^n)$ defined by
	\begin{equation}
		\textstyle S_i(\theta) = \int f_i \theta \ d \mathscr L^n \quad \text{for every $\theta \in \mathscr D(\mathbf R^n)$},
	\end{equation} 
is $\nu_{i,p}$ pointwise differentiable of order $k$ at $\mathscr L^n$ almost every $a \in C_i$. 
We apply \ref{Radamacher theorem for distribution with k jet equals 0} with $T$ replaced by $T-S_i$ 
to conclude that for $\mathscr L^n$ almost every point $a$ of $C_i$ the distribution $T-S_i$ is 
$\nu_{i,p}$ pointwise differentiable of order $k$ at $a$, 
 whence $T$ is $\nu_{i,p}$ pointwise differentiable of order $k$ at $\mathscr L^n$ almost every point of $C_i$.
\end{proof}
	
\begin{remark}
		If $k>0$, then by \cite[4.17]{2021} and \ref{comparison of pointwise differentiability:1} the $k$ jet of $T$ in \ref{Radamacher theorem} must be $0$.
\end{remark}

\appendix
\section{Uniform spaces}\label{uni}

Auxiliary to the proof of \ref{introThmA} (see \ref{Borel-function-of-ptD}), we recall some basic facts about uniform spaces. Uniform spaces can be viewed as a topological generalization of metric spaces in which allow for notions of uniform continuity and equicontinuity (\ref{Basics of uniform space}). The key result required for \ref{Borel-function-of-ptD} is \ref{continuity of fibration} in which we prove a continuity criterion for maps from products of topological spaces into a uniform space. Additionally, we will use a classical projection theorem (\ref{Suslin-projection}) for Borel sets to complete the proof of Borel regularity of the differential (see \ref{ThmA-4} and \ref{ThmA-5}).

\begin{definition}
	\label{Unifor space:Def:Uniformity and entourages}
	\label{Definition of uniform space}\BourbakiGTI{\BourbakiCiteNumber{2}{1}{1} Definition 1}\citep[Ch. \RomanNumeralCaps{2}]{MR910295}
	 A \textit{uniformity} $\mathuniformity U$ on a set $X$ is a family of subsets of $X\times X$ satisfying the following two properties:
	 \begin{enumerate}
	 	\item The intersection of nonempty finite subfamily of $\mathuniformity U$  belongs to $\mathuniformity U$.
	 	\item Every subset $S$ of $X\times X$ containing a member of $\mathuniformity U$ belongs to $\mathuniformity U$.
	 \end{enumerate} 
 	 and also the following three properties:
	\begin{enumerate}\setcounter{enumi}{2}
		\item Every member of $\mathuniformity U$ contains the diagonal $X\times X \cap \{ (x,x) : x\in X\}$.
		\item If $V\in \mathfrak U$ then  $V^{-1}\in \mathuniformity U$.
		\item For each $V\in \mathuniformity U$ there exists $W\in  \mathuniformity U$ such that $W\circ W \subset V$.
	\end{enumerate}
	 A \textit{uniform space} is a set $X$ equipped with a uniformity. An element $U$ of uniformity is called  an \textit{entourage} of the uniform space of $X$.
\end{definition}
\begin{propparagraph}\label{Basics of uniform space}
	There exists a unique topology on $X$ such that for each $x\in X$, the family of sets
	\[ X \cap \{ y : (y,x) \in V\} \quad \text{corresponding to every entourage $V$ of $X$},\]
	is a fundamental system of neighborhoods of the topology.
	The topology just defined  is called the topology induced by uniform structure on $X$ \BourbakiGTI{\BourbakiCiteNumber{2}{1}{2} Proposition 1}.

Given a pseudo metric $
	\rho$ on a set $X$. The uniformity defined by $\rho$ is the uniformity on $X$ characterized by the following property, \textit{ every entourage contains $X\times X \cap \{ (x,y) : \rho(x,y)\leq \epsilon\}$ for some $0<\epsilon<+\infty$} \BourbakiGTII{\BourbakiCiteNumber{9}{1}{2} Definition 2}.
A function $f$ mapping uniform space $X$ into a uniform space $X'$ is \textit{uniformly continuous}
 if and only if 
 the inverse image of an entourage of $X'$ under the map $f\times f$ is an entourage of  $X$ \BourbakiGTI{\BourbakiCiteNumber{2}{2}{1} Definition 3}.
Suppose $X$ is a topological space, $Y$ is a uniform space, and $I$ is a set. Then  a family of functions $(f_i)_{i\in I}$ mapping $X$ in to $Y$ is equicontinuous at $a\in X$ if and only if 
	for every entourage $V$ of $Y$ there exists a neighborhood $U\subset X$ of $a$ such that 
	$
	f_i[U] \subset  Y \cap \{ y : (f_i(a),y) \in V\},
	$
	or equivalently 
	$\bigcap \{ f_i^{-1} [Y \cap \{ y : (f_i(a),y) \in V\}]  : i\in I\}$
	is a neighborhood of $a$. When $I$ is a single point the condition reduces to continuity of a single function.

\end{propparagraph}
\begin{theorem}\label{continuity of fibration}
	Suppose $X,Y$ are topological spaces, $Z$ is a uniform space, and $u $ is a function mapping $X \times Y$ into $Z$ with  the following two properties.
	\begin{enumerate}
		\item Every point $x\in X$ has a neighborhood $S$ such that
		$\{u(x,\cdot) :  x \in S \}$
		is equicontinuous.
		\item For every $b\in Y$, the function $u(\cdot , b)$ is continuous.
	\end{enumerate}
	Then $u$ is continuous.
\end{theorem}

\begin{proof}

	We have to show that $u$ is continuous at every $(a,b)\in X\times Y$. Suppose $U$ is a neighborhood of $u(a,b) = c$ in $Z$ and $\mathscr U$ is the unifomity of $Z$.
	
	Choose $V,V'\in \mathscr U$ with 
	$Z \cap \{z : (z,c)\in V \}\subset U$ and 
	\[ (Z\times Z) \cap\{ (a,b) : (a,z),(z,b) \in V' \quad \text{ for some $z\in Z$} \} \subset V.\]
	By hypothesis, there exists a neighborhood $S$ of $x$ such that $ \{ u(x,\cdot) : x \in S\}$ is equicontinuous at $b\in Y$. Hence, there exists a neighborhood $W$ of $b\in Y$ such that, 
	\begin{gather*}
		(u(x,y) ,u(x,b)) \in  V' \quad \text{for $x\in S$ and $y\in W$}.
	\end{gather*}
	Since $u(\cdot,b)$ is continuous, 
	$ X\cap\{ x : (u(x,b),u(a,b))\in V' \}$ is a neighborhood of $a$. Finally, for every $x\in S $, $y\in W$ with $(u(x,b),u(a,b)) \in V'$, the point
	$(u(x,y),u(x,b))$ and  $(u(x,b),u(a,b))$ belong to $ V'$, hence $(u(x,y),u(a,b))\in V$.
\end{proof}
\begin{definition}\label{Def:Lower semicontinuity}
	Suppose $X$ is a topological space. An $\overline{\mathbf R}$ valued function  $f: X \to \overline{\mathbf R}$ is called $\textit{lower semicontinuous}$ if and only if
	$f^{-1}[\overline{\mathbf R}\cap\{ t : t > \alpha\}] \subset X$ is open for every $\alpha \in \overline{\mathbf R}$.
	Moreover, the following two statements are equivalent:
	\begin{enumerate}
		\item The function $f$ is lower semicontinuous.
		\item The set $(X\times \mathbf R) \cap \{ (x,c) : f(x)\leq c\} $ is closed.
	\end{enumerate}
\end{definition}
\begin{definition}\label{Def:Polish,Lusin, space}
	Suppose $X$ is a topological space. 
	\begin{enumerate}
		\item 	A topological space $X$ is called \textit{metrizable} topological space if and only if there exists a metric inducing the topology of $X$.
		\item  A topological space $X$ is called a $\textit{Polish space}$ if and only if it is separable and  metrizable  by a complete  metric.
		\item  A topological space $X$ is called a $\textit{Lusin space}$ if and only if its topology is Hausdorff and is coarser than a Polish topology.
	\end{enumerate}
\end{definition}

\begin{propparagraph}\label{Suslin-projection}
	We will use the following proposition, which readily follows from \citep[2.2.10]{MR41:1976}:
	\emph{If $X$ is a Lusin space, $Y$ is a Hausdorff space, $f : X \to Y$ is a continuous map, $f|A$ is injective, and $A \subset X$ is a Borel set, then $f(A)$ is a Borel set.}
\end{propparagraph}

\def\MeasurablefunctionDemonstation{}

\section{Locally convex spaces}\label{loc}
In order to study the topology on the space of test functions and the space of distributions, we collect some basic results on locally convex spaces below. The main results which will be used in the proof of theorem \ref{introThmA} are \ref{Montel space:remark:1} and \ref{Topological property of space of distribution}. Additionally, the notions of duality and polar sets (\ref{Duality} and \ref{Def:Weak topology})
will be essential for the proof of \ref{lower-semicontinuous-lemma}, a key ingredient in the proof of theorem \ref{introThmA}

The topology on the dual of a locally convex space induced by uniform convergence on a class of bounded sets $\mathbf B$(\ref{Uniform convergence}) is defined by a set of seminorms related to the bounded sets $\mathbf B$ (\ref{Characterization of uniform convergence for linear functional}).
The notation for uniform convergence in \ref{Uniform convergence} will be used to strengthened the results obtained in \ref{continuity of fibration} to \ref{equi-criteria-special-case}, which is a preliminary lemma to prove the \ref{lower-semicontinuous-lemma}. The main tool for studying topologies of uniform convergence on bounded sets are the duality and polar, which is introduced in \ref{Duality} and \ref{Def:Weak topology}, and will be  also used in the example in \ref{E_c-strictly-stronger-then-E_s}, and comparison of topologies of uniform convergence on bounded sets (\ref{compar-topology}).

Finally, the notions of strong dual, bidual, barreled space, Montel space, and reflexive space (\ref{Def:Reflexivity}) are introduced to study the topology of the space of test functions (\ref{Topological property of space of distribution}). In particular, we show that the topology  on $\mathscr D'(\mathbf R^n)$ induced by uniform convergence on bounded sets and compact sets agree(\ref{Bounded subset of strong topology of space of distribution}). This, together with \ref{Suslin-projection} allow us to improve the result \ref{Borel-function-of-ptD:enum:4} to \ref{ThmA-5}

\begin{definition}\BourbakiTVS{\BourbakiCiteNumber{1}{1}{1}. Definition 1}\label{Def:TVS}
		A topology on a vector space is compatible to the vector space structure if and only if the addition and multiplication is continuous. A \textit{topological vector space} is a vector space $E$ equipped with a topology compatible to the vector space structure of $E$. 
	\end{definition}

\begin{remark}\label{Topology of TVS is uniformizable}
	In particular, the additive group structure on topological vector space $E$ is compatible with its topology and the topology is induced by the uniform structure $\mathuniformity{U}$ consisting of 
	$E \times E \cap \{ (x,y) : x-y\in U\}$ corresponding to every neighborhood $U$ of $0$. A more precise characterization of the local structure is in \BourbakiTVS{\BourbakiCiteNumber{1}{1}{5} Proposition 4}.
\end{remark}
\begin{definition}\BourbakiTVS{\BourbakiCiteNumber{1}{1}{5}. Proposition 4}\label{Def:TVS:LC}\label{Definition of fundamental system of seminorms}
		Suppose $E$ is a topological vector space.	A set $\Gamma$ of real valued seminorms is a \textit{fundamental system of seminorms} if and only if the topology is defined by $\Gamma$. A topological vector space $E$ is \textit{locally convex} if and only if there exist a fundamental system of neighborhood at $0$ consist of convex set. A locally convex space is a topological vector space whose  topology is locally convex.
	\end{definition}

\paragraph{Space of linear maps}

\begin{definition}
	Suppose $E,F$ are topological vector spaces. The set of all linear maps from $E$ to $F$ is denoted by $\Hom(E,F)$. The set $\mathscr L(E,F)$ consists of  all continuous functions in $\Hom(E,F)$. In case $F=\mathbf R$, we let $E^\ast = \Hom(E,\mathbf R),E' = \mathscr L(E,\mathbf R)$. We define a bilinear form $\langle \cdot ,\cdot \rangle : E \times E^\ast \to \mathbf R$ by  $\langle x ,f \rangle = f(x)$.
Suppose $S$ is a subset of locally convex space $E$. 
	The closed convex balanced envelope of $S$ is denoted by $\Gamma(S)$. 
	I.e.\ $\Gamma(S)$ is the smallest closed convex balanced set containing $S$ and 
	 is equal to the closed convex envelope of $S\cup -S$, where$-S = E \cap \{ -s : s\in S\}$.
\end{definition}
\begin{definition}\label{Uniform convergence}
	Suppose $\mathbf B$ is a family of bounded subsets of $E$. For every locally convex space $F$, $\mathscr L_{\mathbf B}(E,F)$ denotes the locally convex space $\mathscr L(E,F)$  equipped with uniform convergence on  $\mathbf B$. The fundamental system of neighborhoods of $0$ is given by 
	$\mathscr L(E,F) \cap \{f : f(S) \subset U \}$
	corresponding to each $S\in \mathbf B$
	 and each neighborhood $U$ of $0\in F$. 
	 In case $F= \mathbf R$, we let $E'_{\mathbf B} = \mathscr L_{\mathbf B}(E,\mathbf R)$
	  and $E_s'$, $E_b'$ and $E_c'$ is the locally convex space $E'_{\mathbf B}$ with $B$ equals the class of all one point, bounded and compact subsets of $E$ respectively.
\end{definition}

\begin{example}
	Suppose $E$ is a normed space. Then the locally convex topology of $E'_b$ is defined by the dual norm. 
\end{example}

\begin{definition}\label{Duality}
	\citep[\RomanNumeralCaps{2}, \S~6.3]{MR910295}Suppose $F,G$ are vector spaces put into duality by a bilinear form $B : F \times G \to \mathbf R$. For every $M \subset F$, the set 
	\[  G\cap \{ y :  B(x,y)   \geq -1 , x\in M\} \]
	is called the $\mathit{polar}$ (with respect to the duality $B$) of $M$ and is denoted by $M^\circ$. If $E$ is a locally convex space, then the polar of a  subset of $E$ means the polar for the duality $\langle \cdot,\cdot \rangle : E \times E' \to \mathbf R$, when the duality with $E$ is not indicated.
\end{definition}
	\begin{definition}\label{Def:Weak topology}
	The weak topology on $F$ defined by a  duality $B : F \times G \to \mathbf R$ is the locally convex topology in  
	which every neighborhood of $0\in F$ contains sets of the form 
	\[ F \cap \bigcap_{i \in I}\{ f : |B(f,g_i)| \leq \epsilon_i  \}  \]
	where $I$ is a finite set, $g_i \in G$, and $0<\epsilon<\infty$.
	 The set $\Gamma(S)^\circ$ is called the \textit{absolute polar} of $S \subset F$ and is equal to
	\[ G \cap \{ g : |B(f,g)|\leq 1 , f \in S \}.\]
	Suppose $E$ is a locally convex space. The weak topology on 
	$E$ and $E'$ defined by the duality $\langle \cdot ,\cdot \rangle$ is denoted by  $\sigma(E,E')$ and $\sigma(E',E)$. 
\end{definition}
\begin{propparagraph}\label{Characterization of uniform convergence for linear functional}
	If $\Sigma$ is a family of seminorms which defines the locally convex topology for $F$, then 
	the topology defined by the family of seminorms
	\[ p_M(f) = \sup \{ p(f(x)) : x\in M\}, \quad \text{for every } f\in \mathscr L_{\mathbf B}(E, F),\]
	corresponding to $M \in \mathbf B $ and $ p \in \Sigma$ agrees with 
	the topology of $\mathscr L_{\mathbf B}(E,F)$ defined in last paragraph(\citep[\RomanNumeralCaps{3}, \S~3.1]{MR910295}). 
	In case $F=\mathbf R$, the fundamental system of neigborhood of $0$ is given by
	$\lambda (\bigcup \Omega)^\circ$ for  each finite subfamily $\Omega$ of 
	$  \{ \Gamma(C) : C \in \mathbf B \}$ and $0<\lambda <+\infty$.
\end{propparagraph}
\begin{remark}\label{family-of-Hyperplane-bounded-over-nbd}
	Suppose $E,F$ are locally convex spaces and $M \subset \Hom(E,F)$. Then $M$ is equicontinuous if and only if for every neighborhood $U$ of  $0$ in $F$, the set $ E \cap \bigcap \{ f^{-1}(U) : f\in M\}$ is neighborhood of $0\in E$.
\end{remark}

\begin{remark}\label{equi}
	Suppose $E$ is a locally convex space, $\mathbf B$ is  a family of bounded subsets of $E$, $S\in \mathbf B$, and $M$ is the set of linear function $f$ such that
	\[ f(\alpha) =  \langle s , \alpha \rangle \quad \text{for every $s\in E'$},\]
	for some $s\in S$. Then $M$ is equicontinuous subset of $E'_{\mathbf B}$.
\end{remark}

\begin{propparagraph}\label{compar-topology}
	\BourbakiTVS{\BourbakiCiteNumber{3}{3}{1}. Proposition 2, \BourbakiCiteNumber{4}{1}{3}. Proposition 6}
	Suppose $E$ is a locally convex spaces, $\mathbf B_1$ and $\mathbf B_2$ are families of bounded subsets of $E$. Then the topology of uniform convergence on $\mathbf B_1$ of $E'$ is stronger than the topology of uniform convergence on  $\mathbf B_2$ of $E'$ if and only if for every $A\in \mathbf B_2$, there exist $0 < \lambda < \infty$ and $U_1,\dots,U_n \in \mathbf B_1$ such that $A \subset \lambda^{-1} \Gamma(\bigcup_{i=0}^n U_i)$ (see \ref{Duality}).
\end{propparagraph}

	\begin{remark}\label{Continuity of embedding the space of distribution}
		Suppose $n$ is a positive integer, $U$ is an open subset of $\mathbf R^n$ and $Y$ is a Banach space.
		Then the inclusion of $\mathscr E'(U,Y)_b$ (resp. $\mathscr E'(U,Y)_c$) into $\mathscr D'(U,Y)_b$ (resp. $\mathscr D'(U,Y)_c$) is continuous.
	\end{remark}
	
\begin{defparagraph}\label{Def:Stong dual and bidual}\label{Def:Barreled}\label{Def:Montel}\label{Def:Reflexivity}
Given a vector space $E$. A subset $A$ of $E$ \textit{absorbs} a subset $B$ of $E$ if and only if there exists a positive real number $\alpha$ such that  for every $\lambda \in \mathbf R$ with $|\lambda| \geq \alpha$, $B \subset \lambda A $. A subset $A$ of $E$ is \textit{absorbent} if and only if it absorbs every singleton set of $E$. A subset $A$ of $E$ is \textit{balanced} if and only if for $\lambda \in \mathbf R$ with $|\lambda|\leq 1$, $\lambda A \subset A$.
	Now, suppose $E$ is a locally convex space.
	The locally convex space $E'_b$ is called the \textit{strong dual} of the locally convex space $E$.
	The strong dual of $E'_b$ is called the \textit{bidual} of the locally convex space $E$. A locally convex space $E$ is \textit{barreled} if and only if every closed convex balanced absorbent subset of $E$ is a neighborhood of $0$ \BourbakiTVS{\BourbakiCiteNumber{3}{4}{1}. Definition 2}, for example
every complete metrizable locally convex space is barreled \BourbakiTVS{\BourbakiCiteNumber{3}{4}{1}. Corollary of Proposition 2}, which includes all Banach spaces. One of the important properties for barreled space is that the Banach Steinhaus theorem holds for barreled spaces \BourbakiTVS{\BourbakiCiteNumber{3}{4}{2}. Theorem 1} . 
		A Hausdorff and barreled locally convex space $E$ is \textit{Montel} if and only if every bounded subset of $E$ is relatively compact \BourbakiTVS{\BourbakiCiteNumber{4}{2}{5}. Definition 4}.
		A locally convex space $E$ is \textit{reflexive} if and only if the canonical mapping $c_E$ from $E$ into its bidual is an isomorphism of topological vector spaces \BourbakiTVS{\BourbakiCiteNumber{4}{2}{3}. Definition 3}.
	\end{defparagraph} 

\begin{remark}\label{Montel space:remark:1}
	For every Montel, Hausdorff, and Barreled locally convex space $E$,
	the topology of $E'_b$ and $E'_c$ agree. 
\end{remark}

\begin{theorem}\BourbakiTVS{\BourbakiCiteNumber{4}{2}{5}. Example 4}\label{Topological property of space of distribution}
		Suppose $n$ is a positive integer, $Y$ is a complete normed space, and $U$ is an open subset of $\mathbf R^n$. Then the locally convex spaces $\mathscr E(U,Y)$ and $\mathscr D(U,Y)$ are Barreled and Hausdorff locally convex spaces. If in additionally $Y$ is finite dimensional then $\mathscr E(U,Y)$ and $\mathscr D(U,Y)$ are Montel, barreled, and Hausdorff locally convex spaces.
	\end{theorem}
\begin{theorem}\label{Bounded subset of strong topology of space of distribution}
	Suppose $n$ is a positive integer, $Y$ is a finite dimensional normed space, and $U$ is an open subset of $\mathbf R^n$.
	In order for a subset $A$ of $\mathscr D'(U,Y)_b$ (resp. $\mathscr E'(U,Y)_b$) to be relatively compact it is necessary and sufficient that
	\[ \sup \{ T(\theta) : T \in A\} <+\infty \quad \text{for every $\theta \in \mathscr D(U,Y)$ (resp. $\mathscr E(U,Y)$)}.\]
\end{theorem}

\begin{proof}
By \ref{Topological property of space of distribution}, $\mathscr E(U,Y)$ and $\mathscr D(U,Y)$ are
 Hausdorff, Barreled, and Montel. By \BourbakiTVS{a statement below \BourbakiCiteNumber{4}{2}{5}. Proposition 8}, $\mathscr E(U,Y)$ and $\mathscr D(U,Y)$ are reflexive.
The conclusion follows from \BourbakiTVS{\BourbakiCiteNumber{4}{2}{5}. Proposition 9}.
\end{proof}


\begin{thebibliography}{11}
\providecommand{\natexlab}[1]{#1}
\providecommand{\url}[1]{\texttt{#1}}
\expandafter\ifx\csname urlstyle\endcsname\relax
  \providecommand{\doi}[1]{doi: #1}\else
  \providecommand{\doi}{doi: \begingroup \urlstyle{rm}\Url}\fi

\bibitem[Agmon(2011)]{Agmon2011}
S.~Agmon.
\newblock \emph{The Lp Approach to the Dirichlet Problem}, pages 49--92.
\newblock Springer Berlin Heidelberg, Berlin, Heidelberg, 2011.
\newblock ISBN 978-3-642-10926-3.
\newblock \doi{10.1007/978-3-642-10926-3_2}.
\newblock URL \url{https://doi.org/10.1007/978-3-642-10926-3_2}.

\bibitem[Bourbaki(1987)]{MR910295}
N.~Bourbaki.
\newblock \emph{Topological vector spaces. Chapters 1--5}.
\newblock Elements of Mathematics (Berlin). Springer-Verlag, Berlin Translated from the French by H. G. Eggleston and S. Madan. URL, 1987.
\newblock URL \url{https://doi.org/10.1007/978-3-642-61715-7}.

\bibitem[Bourbaki(1998{\natexlab{a}})]{bourbaki1998general1}
N.~Bourbaki.
\newblock \emph{General Topology: Chapters 1-4}.
\newblock Elements of mathematics. Springer, 1998{\natexlab{a}}.
\newblock ISBN 9783540642411.
\newblock URL \url{https://books.google.com.tw/books?id=kTFSfmsjDM0C}.

\bibitem[Bourbaki(1998{\natexlab{b}})]{bourbaki1998general2}
N.~Bourbaki.
\newblock \emph{General Topology: Chapters 5--10}.
\newblock Actualit{\'e}s scientifiques et industrielles. Springer Berlin Heidelberg, 1998{\natexlab{b}}.
\newblock ISBN 9783540645634.
\newblock URL \url{https://books.google.com.tw/books?id=bQwhdmL6IjUC}.

\bibitem[Calderón and Zygmund(1961)]{MR0136849}
A.-P. Calderón and A.~Zygmund.
\newblock Local properties of solutions of elliptic partial differential equations.
\newblock \emph{Studia Math}, 20:\penalty0 171--225, 1961.
\newblock URL \url{http://matwbn.icm.edu.pl/ksiazki/sm/sm20/sm20113.pdf}.

\bibitem[Federer(1969)]{MR41:1976}
H.~Federer.
\newblock \emph{Geometric measure theory}.
\newblock Springer, Berlin, 1969, Band 153. Springer-Verlag New York Inc., New York, 1969.

\bibitem[Gilbarg and Trudinger(1998)]{MR1814364}
D.~Gilbarg and N.~S. Trudinger.
\newblock \emph{Elliptic partial differential equations of second order}.
\newblock Classics in Mathematics. Springer-Verlag, Berlin, 2001. Reprint of theedition, 1998.

\bibitem[Menne(2013{\natexlab{a}})]{MR3023856}
U.~Menne.
\newblock Second order rectifiability of integral varifolds of locally bounded first variation.
\newblock \emph{J. Geom. Anal}, 23\penalty0 (2):\penalty0 709--763, 2013{\natexlab{a}}.
\newblock URL \url{https://doi.org/10.1007/s12220-011-9261-5}.

\bibitem[Menne(2013{\natexlab{b}})]{e5d4609f3eb74629ad6704c7dc209467}
U.~Menne.
\newblock Second order rectifiability of integral varifolds of locally bounded first variation.
\newblock \emph{Journal of Geometric Analysis}, 23\penalty0 (2):\penalty0 709--763, Apr. 2013{\natexlab{b}}.
\newblock ISSN 1050-6926.
\newblock \doi{10.1007/s12220-011-9261-5}.
\newblock Funding Information: The author acknowledges financial support via the DFG Forschergruppe 469. The major part of this work was accomplished while the author was at the University of T{\"u}bingen. Some parts were done at the ETH Z{\"u}rich and the work was put in its final form at the AEI Golm. AEI publication number AEI-2008-065.

\bibitem[Menne(2016)]{862d4889d5ff49cf87cff80c14f3769a}
U.~Menne.
\newblock Weakly differentiable functions on varifolds.
\newblock \emph{Indiana University Mathematics Journal}, 65\penalty0 (3):\penalty0 977--988, 2016.
\newblock ISSN 0022-2518.
\newblock \doi{10.1512/iumj.2016.65.5829}.
\newblock Publisher Copyright: {\textcopyright}.

\bibitem[Menne(2021)]{2021}
U.~Menne.
\newblock {Pointwise differentiability of higher-order for distributions}.
\newblock \emph{Analysis and PDE}, 14\penalty0 (2):\penalty0 323 -- 354, 2021.
\newblock \doi{10.2140/apde.2021.14.323}.
\newblock URL \url{https://doi.org/10.2140/apde.2021.14.323}.

\end{thebibliography}
\end{document}